\RequirePackage{amsthm}
\documentclass[%
a4paper
]{mpi2015-cscpreprint}

\usepackage[utf8]{inputenc}
\usepackage{amssymb,amsmath}
\usepackage{mathtools,xcolor,xfrac}

\DeclareFontFamily{U}{mathx}{\hyphenchar\font45}
\DeclareFontShape{U}{mathx}{m}{n}{
      <5> <6> <7> <8> <9> <10>
      <10.95> <12> <14.4> <17.28> <20.74> <24.88>
      mathx10
      }{}
\DeclareSymbolFont{mathx}{U}{mathx}{m}{n}
\DeclareFontSubstitution{U}{mathx}{m}{n}
\DeclareMathAccent{\widecheck}{0}{mathx}{"71}
\DeclareMathAccent{\wideparen}{0}{mathx}{"75}

\newtheorem{theorem}{Theorem}[section]
\newtheorem{lemma}[theorem]{Lemma}
\newtheorem{definition}[theorem]{Definition}

\newtheorem{corollary}[theorem]{Corollary}

\newtheorem{assumption}[theorem]{Assumption}
\newtheorem{remark}[theorem]{Remark}

\Crefname{hypothesis}{Hypothesis}{Hypotheses}

\usepackage{algorithmic}

\usepackage{graphicx,epstopdf}
\usepackage{placeins,afterpage}

\Crefname{ALC@unique}{Line}{Lines}

\usepackage{amsopn}

\usepackage{tikz}  
\usepackage{pgfplots} 
\usepackage{pgfplotstable}
\pgfplotsset{compat=1.13} 
\usetikzlibrary{calc}
\usepgfplotslibrary{external, groupplots, colorbrewer, colormaps}
\tikzset{external/system call = {%
    pdflatex \tikzexternalcheckshellescape
    -halt-on-error
    -interaction=batchmode
    -jobname "\image" "\texsource"}}
\usepackage{filemod}

\usepackage{subcaption}
\usepackage[shortlabels, inline]{enumitem}

\newcommand{\errparam}[2]{%
  \tikzsetnextfilename{err_param_#1}
  \begin{tikzpicture}
    \pgfplotstableread{code/results/figures/#1/#2/%
      piecewise_H2_#1.dat}%
    {\errparamtable}
    \begin{semilogyaxis}[%
        width=0.666\textwidth,
        height=0.1\textheight,
        log basis y=10,
        scale only axis,
        xlabel={Parameter $\p$},
        ylabel={$\varepsilon_{\p}$},
        legend style={
          at={(1.02, 1)},
          anchor=north west,
          font=\footnotesize,
        },
        legend cell align={left},
        legend entries={pIRKA ROM, Opt. ROM, IRKA ROMs},
        cycle list name=hund_three,
        grid=both,
        errs_#1,
      ]

      \foreach \i in {1, 2, 3}
        \addplot table[x index=0, y index=\i] {\errparamtable};
    \end{semilogyaxis}
  \end{tikzpicture}
}

\newcommand{\surfpolefomplots}[1]{%
  \pgfplotstableread{code/results/figures/#1/axes_range_fom.dat}%
  {\axesfomtable}
  \pgfplotstablegetelem{0}{[index]0}\of{\axesfomtable}
  \pgfmathsetmacro{\xmin}{\pgfplotsretval}
  \pgfplotstablegetelem{0}{[index]1}\of{\axesfomtable}
  \pgfmathsetmacro{\xmax}{\pgfplotsretval}
  \pgfplotstablegetelem{1}{[index]0}\of{\axesfomtable}
  \pgfmathsetmacro{\ymin}{\pgfplotsretval}
  \pgfplotstablegetelem{1}{[index]1}\of{\axesfomtable}
  \pgfmathsetmacro{\ymax}{\pgfplotsretval}
  \pgfplotstablegetelem{2}{[index]0}\of{\axesfomtable}
  \pgfmathsetmacro{\zmin}{\pgfplotsretval}
  \pgfplotstablegetelem{2}{[index]1}\of{\axesfomtable}
  \pgfmathsetmacro{\zmax}{\pgfplotsretval}

  \tikzsetnextfilename{surf_pole_#1}
  \begin{tikzpicture}
    \begin{semilogxaxis}[%
        name=surf fom,
        view={0}{90},
        width=\surfsize, 
        height=\surfsize,
        scale only axis,
        enlargelimits=false,
        xlabel={Frequency $\omega$ [rad/s]},
        ylabel style={align=center},
        ylabel={\textbf{FOM}\\Parameter $\p$},
        point meta min=\zmin,
        point meta max=\zmax,
        surf fom colorbar,
        axis on top,
      ]
      \addplot graphics [
        xmin=\xmin, xmax=\xmax,
        ymin=\ymin, ymax=\ymax,
      ] {code/results/figures/#1/surf_H2L2};
    \end{semilogxaxis}
    \begin{axis}[
        poles fom pos,
        poles #1,
        point meta min=\ymin,
        point meta max=\ymax,
        poles colorbar,
        axis on top,
      ]
      \addplot graphics [
        poles #1 limits,
      ] {code/results/figures/#1/poles.pdf};
    \end{axis}
  \end{tikzpicture}
}

\newcommand{\surfpoleromplots}[2]{%
  \pgfplotstableread{code/results/figures/#1/#2/axes_range.dat}%
  {\axesromtable}
  \pgfplotstablegetelem{0}{[index]0}\of{\axesromtable}
  \pgfmathsetmacro{\xmin}{\pgfplotsretval}
  \pgfplotstablegetelem{0}{[index]1}\of{\axesromtable}
  \pgfmathsetmacro{\xmax}{\pgfplotsretval}
  \pgfplotstablegetelem{1}{[index]0}\of{\axesromtable}
  \pgfmathsetmacro{\ymin}{\pgfplotsretval}
  \pgfplotstablegetelem{1}{[index]1}\of{\axesromtable}
  \pgfmathsetmacro{\ymax}{\pgfplotsretval}
  \pgfplotstablegetelem{2}{[index]0}\of{\axesromtable}
  \pgfmathsetmacro{\zmin}{\pgfplotsretval}
  \pgfplotstablegetelem{2}{[index]1}\of{\axesromtable}
  \pgfmathsetmacro{\zmax}{\pgfplotsretval}

  \tikzsetnextfilename{surf_pole_rom_#1}
  \begin{tikzpicture}
    \begin{semilogxaxis}[%
        name=surf pirka,
        view={0}{90},
        width=\surfsize, 
        height=\surfsize,
        scale only axis,
        enlargelimits=false,
        xticklabels={,,},
        ylabel style={align=center},
        ylabel={\textbf{pIRKA ROM}\\Parameter $\p$},
        point meta min=\zmin,
        point meta max=\zmax,
        surf err colorbar,
        axis on top,
      ]
      \addplot graphics [
        xmin=\xmin, xmax=\xmax,
        ymin=\ymin, ymax=\ymax,
      ] {code/results/figures/#1/#2/surf_H2L2_pIRKA};
    \end{semilogxaxis}
    \begin{semilogxaxis}[%
        name=surf granso,
        at={(surf pirka.south west)},
        yshift=-1em,
        anchor=north west,
        view={0}{90},
        width=\surfsize, 
        height=\surfsize,
        scale only axis,
        enlargelimits=false,
        xlabel={Frequency $\omega$ [rad/s]},
        ylabel style={align=center},
        ylabel={\textbf{Opt. ROM}\\Parameter $\p$},
        axis on top,
      ]
      \addplot graphics [
        xmin=\xmin, xmax=\xmax,
        ymin=\ymin, ymax=\ymax,
      ] {code/results/figures/#1/#2/surf_H2L2_GRANSO};
    \end{semilogxaxis}
    \begin{axis}[%
        poles pirka pos,
        poles #1,
        xlabel={},
        xticklabels={,,},
        point meta min=\ymin,
        point meta max=\ymax,
        poles colorbar,
        axis on top,
      ]
      \addplot graphics [
        poles #1 limits,
      ] {code/results/figures/#1/#2/poles_pIRKA.pdf};
    \end{axis}
    \begin{axis}[%
        poles granso pos,
        poles #1,
        axis on top,
      ]
      \addplot graphics [
        poles #1 limits,
      ] {code/results/figures/#1/#2/poles_GRANSO.pdf};
    \end{axis}
  \end{tikzpicture}
}

\newcommand{\surfpoleromplotsextended}[4]{%
  \pgfplotstableread{preliminary_results/figures/#1/#4/axes_range.dat}%
  {\axesromtable}
  \pgfplotstablegetelem{0}{[index]0}\of{\axesromtable}
  \pgfmathsetmacro{\xmin}{\pgfplotsretval}
  \pgfplotstablegetelem{0}{[index]1}\of{\axesromtable}
  \pgfmathsetmacro{\xmax}{\pgfplotsretval}
  \pgfplotstablegetelem{1}{[index]0}\of{\axesromtable}
  \pgfmathsetmacro{\ymin}{\pgfplotsretval}
  \pgfplotstablegetelem{1}{[index]1}\of{\axesromtable}
  \pgfmathsetmacro{\ymax}{\pgfplotsretval}
  \pgfplotstablegetelem{2}{[index]0}\of{\axesromtable}
  \pgfmathsetmacro{\zmin}{\pgfplotsretval}
  \pgfplotstablegetelem{2}{[index]1}\of{\axesromtable}
  \pgfmathsetmacro{\zmax}{\pgfplotsretval}

   \tikzsetnextfilename{surf_pole_rom_#1}
  \begin{tikzpicture}
    \begin{semilogxaxis}[%
        name=surf pirka,
        view={0}{90},
        width=\surfsize, 
        height=\surfsize,
        scale only axis,
        enlargelimits=false,
        xticklabels={,,},
        ylabel style={align=center},
        ylabel={\textbf{pIRKA ROM}\\Parameter $\p$},
        point meta min=\zmin,
        point meta max=\zmax,
        surf err colorbar,
        axis on top,
      ]
      \addplot graphics [
        xmin=\xmin, xmax=\xmax,
        ymin=\ymin, ymax=\ymax,
      ] {code/results/figures/#1/#4/surf_H2L2_pIRKA};
    \end{semilogxaxis}
    \begin{semilogxaxis}[%
        name=surf granso,
        at={(surf pirka.south west)},
        yshift=-1em,
        anchor=north west,
        xticklabels={,,},
        view={0}{90},
        width=\surfsize, 
        height=\surfsize,
        scale only axis,
        enlargelimits=false,
        ylabel style={align=center},
        ylabel={\textbf{Opt. ROM$_{\text{SP}}$}\\Parameter $\p$},
        axis on top,
      ]
      \addplot graphics [
        xmin=\xmin, xmax=\xmax,
        ymin=\ymin, ymax=\ymax,
      ] {code/results/figures/#1/#4/surf_H2L2_GRANSO};
    \end{semilogxaxis}
    \begin{semilogxaxis}[%
        name=surf BC,
        at={(surf granso.south west)},
        yshift=-1em,
        anchor=north west,
        xticklabels={,,},
        view={0}{90},
        width=\surfsize, 
        height=\surfsize,
        scale only axis,
        enlargelimits=false,
        ylabel style={align=center},
        ylabel={\textbf{Opt. ROM$_{\text{IO}}$}\\Parameter $\p$},
        axis on top,
      ]
      \addplot graphics [
        xmin=\xmin, xmax=\xmax,
        ymin=\ymin, ymax=\ymax,
      ] {code/results/figures/#2/#4/surf_H2L2_GRANSO_new_structure};
    \end{semilogxaxis}
    \begin{semilogxaxis}[%
        name=surf EABC,
        at={(surf BC.south west)},
        yshift=-1em,
        anchor=north west,
        view={0}{90},
        width=\surfsize, 
        height=\surfsize,
        scale only axis,
        enlargelimits=false,
        xlabel={Frequency $\omega$ [rad/s]},
        ylabel style={align=center},
        ylabel={\textbf{Opt. ROM$_{\text{All}}$} \\Parameter $\p$},
        axis on top,
      ]
      \addplot graphics [
        xmin=\xmin, xmax=\xmax,
        ymin=\ymin, ymax=\ymax,
      ] {code/results/figures/#3/#4/surf_H2L2_GRANSO_new_structure};
    \end{semilogxaxis}
    \begin{axis}[%
        poles pirka pos,
        poles #1,
        xlabel={},
        xticklabels={,,},
        point meta min=\ymin,
        point meta max=\ymax,
        poles colorbar,
        axis on top,
      ]
      \addplot graphics [
        poles #1 limits,
      ] {code/results/figures/#1/#4/poles_pIRKA.pdf};
    \end{axis}
    \begin{axis}[%
        poles granso pos,
        poles #1,
        xlabel={},
        xticklabels={,,},
        axis on top,
      ]
      \addplot graphics [
        poles #1 limits,
      ] {code/results/figures/#1/#4/poles_GRANSO.pdf};
    \end{axis}
    \begin{axis}[%
        poles BC pos,
        poles #1,
        axis on top,
        xlabel={},
        xticklabels={,,},
      ]
      \addplot graphics [
         poles #1 limits,
      ] {code/results/figures/#2/#4/poles_GRANSO_new_structure.pdf};
    \end{axis}
    \begin{axis}[%
        poles EABC pos,
        poles #1,
        axis on top,
      ]
      \addplot graphics [
        poles #1 limits,
      ] {code/results/figures/#3/#4/poles_GRANSO_new_structure.pdf};
    \end{axis}
  \end{tikzpicture}
}

\newcommand{\errparamextended}[4]{%
   \tikzsetnextfilename{err_param_#1}
  \begin{tikzpicture}
    \pgfplotstableread{code/results/figures/#1/#4/piecewise_H2_#1.dat}{\errparamtable}
    \pgfplotstableread{code/results/figures/#2/#4/piecewise_H2_#2.dat}{\errparamtableBC}
    \pgfplotstableread{code/results/figures/#3/#4/piecewise_H2_#3.dat}{\errparamtableEABC}

    \begin{semilogyaxis}[%
        width=0.666\textwidth,
        height=0.1\textheight,
        log basis y=10,
        scale only axis,
        xlabel={Parameter $\p$},
        ylabel={$\varepsilon_{\p}$},
        legend style={
          at={(1.02, 1)},
          anchor=north west,
          font=\footnotesize,
        },
        legend cell align={left},
       legend entries={pIRKA ROM,
         Opt. ROM$_{\text{SP}}$,
         Opt. ROM$_{\text{IO}}$, 
         Opt. ROM$_{\text{All}}$,
         IRKA ROMs},
        cycle list name=hund_five,
        grid=both,
        errs_#1,
      ]

        \addplot table[x index=0, y index=1] {\errparamtable};
        \addplot table[x index=0, y index=2] {\errparamtable};
        \addplot table[x index=0, y index=2] {\errparamtableBC};
        \addplot table[x index=0, y index=2] {\errparamtableEABC};
        \addplot table[x index=0, y index=3] {\errparamtable};
    \end{semilogyaxis}
  \end{tikzpicture}
}

\pgfplotsset{%
  colormap={dusk}{%
    rgb=(0, 0, 0)
    rgb=(0.0077, 0.0040, 0.0152)
    rgb=(0.0156, 0.0079, 0.0303)
    rgb=(0.0234, 0.0121, 0.0453)
    rgb=(0.0311, 0.0160, 0.0580)
    rgb=(0.0394, 0.0199, 0.0690)
    rgb=(0.0470, 0.0238, 0.0786)
    rgb=(0.0539, 0.0277, 0.0875)
    rgb=(0.0604, 0.0315, 0.0956)
    rgb=(0.0662, 0.0356, 0.1032)
    rgb=(0.0714, 0.0395, 0.1109)
    rgb=(0.0758, 0.0433, 0.1187)
    rgb=(0.0796, 0.0471, 0.1265)
    rgb=(0.0830, 0.0507, 0.1345)
    rgb=(0.0859, 0.0541, 0.1425)
    rgb=(0.0884, 0.0574, 0.1506)
    rgb=(0.0910, 0.0605, 0.1587)
    rgb=(0.0937, 0.0632, 0.1669)
    rgb=(0.0967, 0.0655, 0.1751)
    rgb=(0.0999, 0.0677, 0.1834)
    rgb=(0.1033, 0.0695, 0.1917)
    rgb=(0.1068, 0.0711, 0.2000)
    rgb=(0.1104, 0.0726, 0.2085)
    rgb=(0.1139, 0.0742, 0.2169)
    rgb=(0.1176, 0.0756, 0.2253)
    rgb=(0.1213, 0.0770, 0.2338)
    rgb=(0.1252, 0.0784, 0.2422)
    rgb=(0.1292, 0.0798, 0.2506)
    rgb=(0.1333, 0.0811, 0.2591)
    rgb=(0.1375, 0.0823, 0.2675)
    rgb=(0.1414, 0.0836, 0.2761)
    rgb=(0.1454, 0.0848, 0.2847)
    rgb=(0.1494, 0.0860, 0.2933)
    rgb=(0.1536, 0.0872, 0.3020)
    rgb=(0.1580, 0.0882, 0.3106)
    rgb=(0.1625, 0.0892, 0.3191)
    rgb=(0.1672, 0.0901, 0.3276)
    rgb=(0.1719, 0.0909, 0.3362)
    rgb=(0.1769, 0.0917, 0.3447)
    rgb=(0.1820, 0.0923, 0.3532)
    rgb=(0.1873, 0.0928, 0.3616)
    rgb=(0.1927, 0.0931, 0.3701)
    rgb=(0.1973, 0.0939, 0.3789)
    rgb=(0.2017, 0.0947, 0.3879)
    rgb=(0.2060, 0.0955, 0.3969)
    rgb=(0.2101, 0.0964, 0.4061)
    rgb=(0.2138, 0.0975, 0.4155)
    rgb=(0.2171, 0.0987, 0.4250)
    rgb=(0.2198, 0.1002, 0.4348)
    rgb=(0.2221, 0.1018, 0.4447)
    rgb=(0.2198, 0.1060, 0.4556)
    rgb=(0.2148, 0.1116, 0.4666)
    rgb=(0.2088, 0.1173, 0.4776)
    rgb=(0.2016, 0.1234, 0.4885)
    rgb=(0.1932, 0.1297, 0.4993)
    rgb=(0.1832, 0.1362, 0.5100)
    rgb=(0.1717, 0.1429, 0.5204)
    rgb=(0.1583, 0.1498, 0.5307)
    rgb=(0.1432, 0.1571, 0.5391)
    rgb=(0.1293, 0.1648, 0.5446)
    rgb=(0.1143, 0.1723, 0.5496)
    rgb=(0.0983, 0.1795, 0.5541)
    rgb=(0.0810, 0.1866, 0.5582)
    rgb=(0.0614, 0.1934, 0.5618)
    rgb=(0.0381, 0.2000, 0.5650)
    rgb=(0.0127, 0.2066, 0.5680)
    rgb=(0, 0.2129, 0.5704)
    rgb=(0, 0.2189, 0.5718)
    rgb=(0, 0.2247, 0.5731)
    rgb=(0, 0.2305, 0.5742)
    rgb=(0, 0.2363, 0.5752)
    rgb=(0, 0.2420, 0.5764)
    rgb=(0, 0.2476, 0.5774)
    rgb=(0, 0.2533, 0.5784)
    rgb=(0, 0.2588, 0.5793)
    rgb=(0, 0.2643, 0.5798)
    rgb=(0, 0.2697, 0.5803)
    rgb=(0, 0.2750, 0.5807)
    rgb=(0, 0.2804, 0.5811)
    rgb=(0, 0.2857, 0.5816)
    rgb=(0, 0.2911, 0.5820)
    rgb=(0, 0.2964, 0.5825)
    rgb=(0, 0.3017, 0.5829)
    rgb=(0, 0.3069, 0.5829)
    rgb=(0, 0.3120, 0.5828)
    rgb=(0, 0.3171, 0.5827)
    rgb=(0, 0.3223, 0.5826)
    rgb=(0, 0.3274, 0.5825)
    rgb=(0, 0.3325, 0.5824)
    rgb=(0, 0.3377, 0.5823)
    rgb=(0, 0.3428, 0.5822)
    rgb=(0, 0.3478, 0.5817)
    rgb=(0, 0.3527, 0.5808)
    rgb=(0, 0.3576, 0.5800)
    rgb=(0, 0.3625, 0.5791)
    rgb=(0, 0.3675, 0.5782)
    rgb=(0, 0.3724, 0.5773)
    rgb=(0, 0.3773, 0.5764)
    rgb=(0, 0.3822, 0.5755)
    rgb=(0, 0.3870, 0.5743)
    rgb=(0, 0.3916, 0.5722)
    rgb=(0, 0.3962, 0.5701)
    rgb=(0, 0.4007, 0.5680)
    rgb=(0, 0.4052, 0.5658)
    rgb=(0, 0.4098, 0.5636)
    rgb=(0, 0.4143, 0.5612)
    rgb=(0, 0.4188, 0.5588)
    rgb=(0, 0.4232, 0.5563)
    rgb=(0, 0.4270, 0.5528)
    rgb=(0, 0.4308, 0.5493)
    rgb=(0.0311, 0.4346, 0.5458)
    rgb=(0.0736, 0.4383, 0.5421)
    rgb=(0.1039, 0.4420, 0.5385)
    rgb=(0.1290, 0.4455, 0.5348)
    rgb=(0.1512, 0.4490, 0.5311)
    rgb=(0.1714, 0.4525, 0.5274)
    rgb=(0.1953, 0.4552, 0.5245)
    rgb=(0.2178, 0.4578, 0.5218)
    rgb=(0.2386, 0.4603, 0.5194)
    rgb=(0.2582, 0.4627, 0.5171)
    rgb=(0.2769, 0.4650, 0.5151)
    rgb=(0.2947, 0.4672, 0.5134)
    rgb=(0.3119, 0.4694, 0.5118)
    rgb=(0.3285, 0.4715, 0.5105)
    rgb=(0.3446, 0.4735, 0.5095)
    rgb=(0.3603, 0.4755, 0.5087)
    rgb=(0.3757, 0.4774, 0.5082)
    rgb=(0.3906, 0.4792, 0.5079)
    rgb=(0.4053, 0.4810, 0.5079)
    rgb=(0.4197, 0.4827, 0.5081)
    rgb=(0.4339, 0.4844, 0.5085)
    rgb=(0.4478, 0.4860, 0.5092)
    rgb=(0.4615, 0.4875, 0.5100)
    rgb=(0.4750, 0.4890, 0.5110)
    rgb=(0.4884, 0.4904, 0.5121)
    rgb=(0.5016, 0.4918, 0.5135)
    rgb=(0.5147, 0.4932, 0.5148)
    rgb=(0.5277, 0.4945, 0.5161)
    rgb=(0.5405, 0.4958, 0.5175)
    rgb=(0.5533, 0.4970, 0.5190)
    rgb=(0.5659, 0.4982, 0.5205)
    rgb=(0.5785, 0.4994, 0.5220)
    rgb=(0.5910, 0.5005, 0.5235)
    rgb=(0.6034, 0.5016, 0.5249)
    rgb=(0.6158, 0.5027, 0.5264)
    rgb=(0.6281, 0.5037, 0.5278)
    rgb=(0.6403, 0.5047, 0.5292)
    rgb=(0.6525, 0.5056, 0.5305)
    rgb=(0.6647, 0.5065, 0.5317)
    rgb=(0.6768, 0.5074, 0.5329)
    rgb=(0.6888, 0.5083, 0.5340)
    rgb=(0.7008, 0.5092, 0.5349)
    rgb=(0.7128, 0.5100, 0.5357)
    rgb=(0.7247, 0.5108, 0.5364)
    rgb=(0.7365, 0.5116, 0.5369)
    rgb=(0.7483, 0.5124, 0.5372)
    rgb=(0.7599, 0.5133, 0.5373)
    rgb=(0.7716, 0.5141, 0.5372)
    rgb=(0.7831, 0.5149, 0.5369)
    rgb=(0.7945, 0.5158, 0.5364)
    rgb=(0.8058, 0.5167, 0.5357)
    rgb=(0.8170, 0.5177, 0.5347)
    rgb=(0.8280, 0.5187, 0.5335)
    rgb=(0.8389, 0.5197, 0.5322)
    rgb=(0.8497, 0.5209, 0.5305)
    rgb=(0.8602, 0.5222, 0.5285)
    rgb=(0.8703, 0.5237, 0.5260)
    rgb=(0.8803, 0.5253, 0.5234)
    rgb=(0.8901, 0.5270, 0.5205)
    rgb=(0.8996, 0.5289, 0.5175)
    rgb=(0.9089, 0.5308, 0.5143)
    rgb=(0.9180, 0.5329, 0.5110)
    rgb=(0.9269, 0.5351, 0.5075)
    rgb=(0.9347, 0.5381, 0.5036)
    rgb=(0.9388, 0.5432, 0.4996)
    rgb=(0.9428, 0.5484, 0.4957)
    rgb=(0.9466, 0.5537, 0.4918)
    rgb=(0.9501, 0.5592, 0.4880)
    rgb=(0.9536, 0.5647, 0.4842)
    rgb=(0.9568, 0.5702, 0.4805)
    rgb=(0.9599, 0.5759, 0.4767)
    rgb=(0.9629, 0.5816, 0.4730)
    rgb=(0.9637, 0.5884, 0.4711)
    rgb=(0.9644, 0.5952, 0.4691)
    rgb=(0.9650, 0.6020, 0.4672)
    rgb=(0.9655, 0.6088, 0.4654)
    rgb=(0.9661, 0.6155, 0.4634)
    rgb=(0.9666, 0.6223, 0.4615)
    rgb=(0.9670, 0.6291, 0.4597)
    rgb=(0.9673, 0.6359, 0.4579)
    rgb=(0.9674, 0.6427, 0.4564)
    rgb=(0.9673, 0.6496, 0.4551)
    rgb=(0.9671, 0.6565, 0.4538)
    rgb=(0.9668, 0.6634, 0.4526)
    rgb=(0.9664, 0.6703, 0.4515)
    rgb=(0.9659, 0.6773, 0.4505)
    rgb=(0.9652, 0.6842, 0.4496)
    rgb=(0.9645, 0.6911, 0.4488)
    rgb=(0.9637, 0.6981, 0.4482)
    rgb=(0.9626, 0.7051, 0.4478)
    rgb=(0.9615, 0.7121, 0.4476)
    rgb=(0.9602, 0.7191, 0.4474)
    rgb=(0.9589, 0.7261, 0.4475)
    rgb=(0.9574, 0.7331, 0.4477)
    rgb=(0.9558, 0.7401, 0.4480)
    rgb=(0.9540, 0.7471, 0.4486)
    rgb=(0.9520, 0.7542, 0.4498)
    rgb=(0.9494, 0.7613, 0.4522)
    rgb=(0.9468, 0.7685, 0.4548)
    rgb=(0.9441, 0.7756, 0.4576)
    rgb=(0.9412, 0.7828, 0.4608)
    rgb=(0.9382, 0.7899, 0.4643)
    rgb=(0.9351, 0.7970, 0.4680)
    rgb=(0.9318, 0.8041, 0.4721)
    rgb=(0.9285, 0.8112, 0.4771)
    rgb=(0.9267, 0.8173, 0.4883)
    rgb=(0.9250, 0.8234, 0.4995)
    rgb=(0.9233, 0.8294, 0.5109)
    rgb=(0.9219, 0.8353, 0.5225)
    rgb=(0.9206, 0.8412, 0.5344)
    rgb=(0.9194, 0.8469, 0.5464)
    rgb=(0.9184, 0.8526, 0.5585)
    rgb=(0.9174, 0.8582, 0.5708)
    rgb=(0.9177, 0.8634, 0.5836)
    rgb=(0.9182, 0.8685, 0.5965)
    rgb=(0.9188, 0.8735, 0.6094)
    rgb=(0.9196, 0.8784, 0.6224)
    rgb=(0.9205, 0.8833, 0.6354)
    rgb=(0.9216, 0.8881, 0.6484)
    rgb=(0.9228, 0.8929, 0.6615)
    rgb=(0.9242, 0.8976, 0.6746)
    rgb=(0.9260, 0.9021, 0.6876)
    rgb=(0.9280, 0.9066, 0.7007)
    rgb=(0.9302, 0.9110, 0.7138)
    rgb=(0.9325, 0.9154, 0.7269)
    rgb=(0.9349, 0.9197, 0.7400)
    rgb=(0.9374, 0.9240, 0.7530)
    rgb=(0.9400, 0.9283, 0.7661)
    rgb=(0.9427, 0.9325, 0.7791)
    rgb=(0.9455, 0.9367, 0.7922)
    rgb=(0.9486, 0.9408, 0.8052)
    rgb=(0.9518, 0.9449, 0.8182)
    rgb=(0.9550, 0.9489, 0.8312)
    rgb=(0.9583, 0.9529, 0.8441)
    rgb=(0.9616, 0.9569, 0.8571)
    rgb=(0.9649, 0.9609, 0.8701)
    rgb=(0.9683, 0.9649, 0.8831)
    rgb=(0.9718, 0.9688, 0.8960)
    rgb=(0.9753, 0.9728, 0.9090)
    rgb=(0.9788, 0.9767, 0.9220)
    rgb=(0.9824, 0.9806, 0.9350)
    rgb=(0.9860, 0.9844, 0.9480)
    rgb=(0.9894, 0.9884, 0.9610)
    rgb=(0.9929, 0.9923, 0.9740)
    rgb=(0.9965, 0.9961, 0.9870)
    rgb=(1, 1, 1)
  },
}

\tikzstyle{hund_common}=[
  only marks,
  every mark/.append style={solid},
  mark size=0.05ex,
  mark=*,
]

\pgfplotscreateplotcyclelist{hund}{%
  mycolor1, hund_common\\%
  mycolor2, hund_common\\%
  mycolor3, hund_common\\%
  mycolor4, hund_common\\%
  mycolor5, hund_common\\%
  mycolor6, hund_common\\%
  mycolor7, hund_common\\%
}

\tikzstyle{hund_thicker_common}=[
  only marks,
  every mark/.append style={solid},
  mark=*,
]

\pgfplotscreateplotcyclelist{hund_thicker}{%
  mycolor1, mark size=2.5pt, hund_thicker_common\\%
  mycolor2, mark size=2pt, hund_thicker_common\\%
  mycolor3, mark size=1.5pt, hund_thicker_common\\%
  mycolor4, mark size=1pt, hund_thicker_common\\%
  mycolor5, mark size=0.5pt, hund_thicker_common\\%
}

\pgfplotscreateplotcyclelist{hund_solid}{%
  mycolor1, solid, line width=0.2ex\\%
  mycolor2, solid, line width=0.2ex\\%
  mycolor3, solid, line width=0.2ex\\%
  mycolor4, solid, line width=0.2ex\\%
  mycolor5, solid, line width=0.2ex\\%
  mycolor6, solid, line width=0.2ex\\%
  mycolor7, solid, line width=0.2ex\\%
}

\tikzstyle{hund_three_common}=[
  line width=0.2ex,
  mark size=0.5ex,
]

\pgfplotscreateplotcyclelist{hund_three}{%
  mycolor1, hund_three_common, mark=x, mark repeat=3\\%
  mycolor2, hund_three_common, mark=o, mark repeat=3, mark phase=2\\%
  mycolor3, hund_three_common, mark=triangle, mark repeat=3, mark phase=3\\%
}
\pgfplotscreateplotcyclelist{hund_five}{%
  mycolor1, hund_three_common, mark=x, mark repeat=3\\%
  mycolor2, hund_three_common, mark=o, mark repeat=3, mark phase=2\\%
  mycolor3, hund_three_common, mark=triangle, mark repeat=3, mark phase=3\\%
  mycolor4, hund_three_common, mark=x, mark repeat=3\\%
  mycolor5, hund_three_common, mark=o, mark repeat=3, mark phase=2\\%
}

\newlength{\surfsize}
\newlength{\propsize}
\pgfplotsset{
  every tick label/.append style={font=\footnotesize},
  scaled ticks=false,
  ticklabel style={
    /pgf/number format/fixed,
    /pgf/number format/precision=5,
  },
  colorbar top/.style={
    colorbar horizontal,
    colorbar style={
      at={(0.5, 1.05)},
      anchor=south,
      xticklabel pos=upper,
      ylabel style={rotate=-90},
    },
  },
  surf fom colorbar/.style={
    colorbar top,
    colormap/bone,
    colorbar style={
      xticklabel=$10^{\pgfmathparse{\tick}\pgfmathprintnumber\pgfmathresult}$,
      ylabel=$\abs{H(\imag \omega; \p)}$,
    },
  },
  surf err colorbar/.style={
    colorbar top,
    colormap name=dusk,
    colorbar style={
      xticklabel=$10^{\pgfmathparse{\tick}\pgfmathprintnumber\pgfmathresult}$,
      ylabel=\(\varepsilon_{\omega, \p}\),
    },
  },
  poles colorbar/.style={
    colorbar top,
    colormap/viridis,
    colorbar style={
      ylabel=\(\p\),
    },
  },
  poles/.style={
    width=\propsize,
    height=\propsize,
    scale only axis,
    enlargelimits=false,
    xlabel={Real part},
    ylabel={Imaginary part},
  },
  poles Synthetic/.style={
    poles,
    x tick label style={/pgf/number format/1000 sep=},
    y tick label style={/pgf/number format/1000 sep=},
    ylabel shift=-2.5ex,
  },
  plot graphics/poles Synthetic limits/.style={
    xmin=-1000,
    xmax=0,
    ymin=-1000,
    ymax=1000,
  },
  poles Penzl/.style={
    poles,
    xtick={-3, -2, -1, 0},
    xticklabels={$-10^3$, $-10^2$, $-10^1$, $-10^0$},
  },
  plot graphics/poles Penzl limits/.style={
    xmin=-3.3,
    xmax=0.3,
    ymin=-480,
    ymax=480,
  },
  poles TripleChainParametricDamping/.style={
    poles,
  },
  plot graphics/poles TripleChainParametricDamping limits/.style={
    xmin=-0.011,
    xmax=0,
    ymin=-0.02,
    ymax=0.02,
  },
  poles pirka pos/.style={
    at={(surf pirka.north east)},
    xshift=5em,
    anchor=north west,
  },
  poles granso pos/.style={
    at={(surf granso.north east)},
    xshift=5em,
    anchor=north west,
  },
  poles BC pos/.style={
    at={(surf BC.north east)},
    xshift=5em,
    anchor=north west,
  },
  poles EABC pos/.style={
    at={(surf EABC.north east)},
    xshift=5em,
    anchor=north west,
  },
  poles fom pos/.style={
    at={(surf fom.north east)},
    xshift=5em,
    anchor=north west,
  },
  errs_Synthetic/.style={
    xmin=-0.05,
    xmax=1.05,
    minor x tick num=1,
    ymin=1e-9,
    ymax=10,
    ytickten={-8,-7,...,0},
    yticklabels={$10^{-8}$, , $10^{-6}$, , $10^{-4}$, , $10^{-2}$, , $10^{0}$},
  },
  errs_Penzl/.style={
    xmin=5,
    xmax=105,
    minor x tick num=1,
    ymin=1e-5,
    ymax=1,
    ytickten={-5,-4,...,0},
  },
  errs_TripleChainParametricDamping/.style={
    xmin=0.0012,
    xmax=0.0208,
    xtick={0,0.004, 0.008, 0.012, 0.016, 0.02},
    minor x tick num=1,
    ymin=1e-2,
    ymax=10,
    ytickten={-2,-1,0,1,2},
  },
}

\newcommand{\granso}{GRANSO}
\newcommand{\matlab}{MATLAB}
\newcommand{\mmess}{\mbox{M-M.E.S.S.}}

\DeclareMathOperator{\dif}{d\!}

\newcommand{\p}{\mathsf{p}}
\newcommand{\pset}{\mathcal{P}}
\newcommand{\measure}{\mu}

\newcommand{\chA}{\check{A}}
\newcommand{\chB}{\check{B}}
\newcommand{\chC}{\check{C}}
\newcommand{\chE}{\check{E}}

\newcommand{\hA}{\hat{A}}
\newcommand{\hB}{\hat{B}}
\newcommand{\hC}{\hat{C}}
\newcommand{\hE}{\hat{E}}
\newcommand{\hG}{\hat{G}}
\newcommand{\hH}{\hat{H}}
\newcommand{\hP}{\hat{P}}
\newcommand{\hQ}{\hat{Q}}
\newcommand{\hSigma}{\hat{\Sigma}}
\newcommand{\hV}{\hat{V}}
\newcommand{\hW}{\hat{W}}
\newcommand{\ha}{\hat{a}}
\newcommand{\hb}{\hat{b}}
\newcommand{\hc}{\hat{c}}
\newcommand{\he}{\hat{e}}

\newcommand{\hx}{\hat{x}}
\newcommand{\hy}{\hat{y}}

\newcommand{\cB}{\mathcal{B}}
\newcommand{\cC}{\mathcal{C}}
\newcommand{\cD}{\mathcal{D}}
\newcommand{\cH}{\mathcal{H}}
\newcommand{\cJ}{\mathcal{J}}
\newcommand{\cK}{\mathcal{K}}
\newcommand{\cL}{\mathcal{L}}
\newcommand{\cM}{\mathcal{M}}
\newcommand{\cR}{\mathcal{R}}

\newcommand{\tP}{\tilde{P}}
\newcommand{\tQ}{\tilde{Q}}
\newcommand{\tV}{\tilde{V}}
\newcommand{\tW}{\tilde{W}}

\newcommand{\tcJ}{\tilde{\cJ}}

\newcommand{\Ap}{A(\p)}
\newcommand{\Bp}{B(\p)}
\newcommand{\Cp}{C(\p)}

\newcommand{\Ep}{E(\p)}
\newcommand{\Pp}{P(\p)}
\newcommand{\Qp}{Q(\p)}
\newcommand{\Sigmap}{\Sigma_{\pset}}

\newcommand{\cDp}{\cD(\p)}

\newcommand{\hAp}{\hA(\p)}
\newcommand{\hBp}{\hB(\p)}
\newcommand{\hCp}{\hC(\p)}
\newcommand{\hEp}{\hE(\p)}
\newcommand{\hPp}{\hP(\p)}
\newcommand{\hQp}{\hQ(\p)}
\newcommand{\hSigmap}{\hSigma_{\pset}}

\newcommand{\tPp}{\tP(\p)}
\newcommand{\tQp}{\tQ(\p)}

\newcommand{\Htwo}{\cH_2}
\newcommand{\Ltwo}{\cL_2}
\newcommand{\Linf}{\cL_{\infty}}
\newcommand{\HoL}{{\Htwo \otimes \Ltwo}}

\newcommand{\intp}{\int_{\pset}}
\newcommand{\intz}{\int_{0}^{1}}

\newcommand{\bbR}{\mathbb{R}}
\newcommand{\Rn}{\bbR^{n}}
\newcommand{\Rm}{\bbR^{m}}
\newcommand{\Rp}{\bbR^{p}}
\newcommand{\Rd}{\bbR^{d}}

\newcommand{\Rnn}{\bbR^{n \times n}}
\newcommand{\Rnr}{\bbR^{n \times r}}
\newcommand{\Rnm}{\bbR^{n \times m}}
\newcommand{\Rpn}{\bbR^{p \times n}}

\newcommand{\Rrr}{\bbR^{r \times r}}
\newcommand{\Rrm}{\bbR^{r \times m}}
\newcommand{\Rpr}{\bbR^{p \times r}}

\newcommand{\bbC}{\mathbb{C}}
\newcommand{\Cpm}{\bbC^{p \times m}}

\newcommand{\bJ}{\cJ_{\mathrm{s}}}
\newcommand{\imag}{\mathbf{i}}

\newcommand{\tran}{\ensuremath{^{\mathsf{T}}}}
\newcommand{\mtran}{\ensuremath{^{-\!\mathsf{T}}}}

\newcommand{\fundef}[3]{#1 \colon #2 \to #3}

\DeclarePairedDelimiter{\myparen}{\lparen}{\rparen}
\DeclarePairedDelimiter{\mybrack}{\lbrack}{\rbrack}

\DeclarePairedDelimiter{\card}{\lvert}{\rvert}
\DeclarePairedDelimiter{\abs}{\lvert}{\rvert}
\DeclarePairedDelimiter{\norm}{\lVert}{\rVert}
\DeclarePairedDelimiterXPP{\normHtwo}[1]{}{\lVert}{\rVert}{_{\Htwo}}{#1}
\DeclarePairedDelimiterXPP{\normLtwo}[1]{}{\lVert}{\rVert}{_{\Ltwo}}{#1}
\DeclarePairedDelimiterXPP{\normLinf}[1]{}{\lVert}{\rVert}{_{\Linf}}{#1}
\DeclarePairedDelimiterXPP{\normHoL}[1]{}{\lVert}{\rVert}{_{\HoL}}{#1}
\DeclarePairedDelimiterXPP{\normF}[1]{}{\lVert}{\rVert}{_{\operatorname{F}}}{#1}
\DeclarePairedDelimiterXPP{\ip}[1]{}{\langle}{\rangle}{}{#1}
\DeclarePairedDelimiterXPP{\ipF}[2]{}{\langle}{\rangle}{_{\operatorname{F}}}{#1, #2}
\DeclarePairedDelimiterXPP{\Real}[1]{\operatorname{Re}}{\lparen}{\rparen}{}{#1}
\DeclarePairedDelimiterXPP{\eig}[1]{\Lambda}{\lparen}{\rparen}{}{#1}
\DeclarePairedDelimiterXPP{\lo}[1]{o}{\lparen}{\rparen}{}{#1}
\DeclarePairedDelimiterXPP{\bo}[1]{\mathcal{O}}{\lparen}{\rparen}{}{#1}
\DeclarePairedDelimiterXPP{\myvec}[1]{\operatorname{vec}}{\lparen}{\rparen}{}{#1}
\DeclarePairedDelimiterXPP{\spab}[1]{\alpha}{\lparen}{\rparen}{}{#1}
\DeclarePairedDelimiterXPP{\trace}[1]{\operatorname{tr}}{\lparen}{\rparen}{}{#1}

\let\hat\widehat%
\let\tilde\widetilde%
\let\check\widecheck%

\definecolor{mycolor1}{HTML}{1F77B4}
\definecolor{mycolor2}{HTML}{FF7F0E}
\definecolor{mycolor3}{HTML}{2CA02C}
\definecolor{mycolor4}{HTML}{D62728}
\definecolor{mycolor5}{HTML}{9467BD}
\definecolor{mycolor6}{HTML}{8C564B}
\definecolor{mycolor7}{HTML}{E377C2}
\definecolor{mycolor8}{HTML}{7F7F7F}
\definecolor{mycolor9}{HTML}{BCBD22}
\definecolor{mycolor10}{HTML}{17BECF}

\usepackage{xargs}
\usepackage[
 disable
]{todonotes}

\makeatletter
\renewcommand{\todo}[2][]{\tikzexternaldisable\@todo[#1]{#2}\tikzexternalenable}
\makeatother

\newcommandx{\manu}[2][1=]{\todo[linecolor=magenta,size=\tiny,
backgroundcolor=magenta!25,bordercolor=magenta,#1]{MH: #2}}
\newcommandx{\jens}[2][1=]{\todo[linecolor=black,size=\tiny,
backgroundcolor=black!25,bordercolor=black,#1]{JS: #2}}
\newcommandx{\petar}[2][1=]{\todo[linecolor=blue,size=\tiny,
backgroundcolor=blue!25,bordercolor=blue,#1]{PM: #2}}
\newcommandx{\tim}[2][1=]{\todo[linecolor=red,backgroundcolor=red!25,
bordercolor=red,size=\tiny,#1]{TM: #2}}

\ifpdf
\hypersetup{
  pdftitle={Optimization-based parametric model order reduction via
  H2xL2 first-order necessary conditions},
  pdfauthor={Manuela Hund, Tim Mitchell, Petar Mlinaric, and Jens Saak}
}
\fi

\begin{document}
\title{Optimization-based parametric model order reduction via
  $\HoL$ first-order necessary conditions}

\author[$\dagger$]{Manuela~Hund}
\author[$\dagger$]{Tim~Mitchell}
\author[$\dagger$]{Petar~Mlinari\'{c}}
\author[$\dagger$]{Jens~Saak}
\affil[$\dagger$]{%
  Max Planck Institute for Dynamics of Complex Technical Systems,
  Sandtorstra{\ss}e 1, 39106 Magdeburg, Germany\authorcr{}
  (\email{hund@mpi-magdeburg.mpg.de}, \orcid{0000-0003-2888-3717},
  \email{mitchell@mpi-magdeburg.mpg.de}, \orcid{0000-0002-8426-0242},
  \email{mlinaric@mpi-magdeburg.mpg.de}, \orcid{0000-0002-9437-7698},
  \email{saak@mpi-magdeburg.mpg.de}, \orcid{0000-0001-5567-9637})
}
\shorttitle{$\HoL$ Parametric MOR via Optimization}
\shortauthor{M. Hund, T. Mitchell, P. Mlinari\'{c}, and J. Saak}
\shortdate{2021-03-04}

\abstract{%
  In this paper, we generalize existing frameworks for $\HoL$-optimal model
  order reduction to a broad class of parametric linear time-invariant systems.
  To this end, we derive first-order necessary optimality conditions for a class
  of structured reduced-order models, and then building on those,
  propose a stability-preserving optimization-based method for computing locally
  $\HoL$-optimal reduced-order models.
  We also make a theoretical comparison to existing approaches in the literature,
  and in numerical experiments, 
  show how our new method, with reasonable computational effort, 
  produces stable optimized reduced-order models with significantly lower 
  approximation errors.
}



\maketitle


\section{Introduction}
Given parameters $\p = \myparen{\p_1, \dots, \p_d} \in \pset \subset \Rd$,
we consider parametric linear time-invariant systems
with time-domain state-space realizations
\begin{equation}
  \label{eq:pLTI}
  \begin{aligned}
    \Ep \dot{x}(t; \p) & = \Ap x(t; \p) + \Bp u(t), \quad x(0; \p) = 0, \\
    y(t; \p) & = \Cp x(t; \p),
  \end{aligned}
\end{equation}
where $t \geq 0$ is the time and
$H$ is the associated (parameterized) transfer function in the Laplace domain,
i.e., $H(s; \p) = \Cp \myparen{s \Ep - \Ap}^{-1} \Bp \in \Cpm$
with $s \in \bbC$ being the Laplace variable.
Matrices $\Ep, \Ap \in \Rnn$ describe the system dynamics,
while $\Bp \in \Rnm$ is the input matrix,
$\Cp \in \Rpn$ the output matrix,
$x(t; \p) \in \Rn$ the state,
$u(t) \in \Rm$ the input, and
$y(t; \p) \in \Rp$ the output.
For all parameter values $\p \in \pset$,
we assume that
$\Ep$ is invertible and
system~\eqref{eq:pLTI} is \emph{asymptotically stable}, i.e.,
the set of eigenvalues of the matrix pencil $\lambda \Ep - \Ap$,
which we denote as $\eig{\Ap, \Ep}$,
is in the open left half of the complex plane.
Defining
\begin{equation*}
  \spab{\Ap, \Ep} \coloneqq
  \max\{\Real{\lambda} : \lambda \in \eig{\Ap, \Ep}\},
\end{equation*}
i.e., the \emph{spectral abscissa} of matrix pencil $\lambda \Ep - \Ap$,
the stability assumption can equivalently be written as
$\max_{\p \in \pset} \spab{\Ap, \Ep} < 0$ under mild assumptions (see~\cref{thm:assumptions}
in~\cref{sec:preliminaries}).

Such parametric systems arise in applications where the parameters are used to
describe things like geometric or physical properties.
In practice, for fidelity, the state-space dimension $n$ is typically large, and
so corresponding numerical calculations can be prohibitively expensive, even in
the non-parametric case.
One approach to dealing with this is parametric model order reduction (MOR), i.e.,
constructing a parametric reduced-order model (ROM)
\begin{equation}
  \label{eq:pLTIr}
  \begin{aligned}
    \hEp \dot{\hx}(t; \p) & = \hAp \hx(t; \p) + \hBp u(t), \quad
    \hx(0; \p) = 0, \\
    \hy(t; \p) & = \hCp \hx(t; \p),
  \end{aligned}
\end{equation}
that has similar dynamical behavior to the full-order model
(FOM)~\eqref{eq:pLTI} across the entire parameter domain $\pset$ but with
a much smaller dimension $r \ll n$ and
matrix-valued functions $\hE, \hA, \hB, \hC$ that are cheap to evaluate.
For~\eqref{eq:pLTIr},
$\hEp, \hAp \in \Rrr$,
$\hBp \in \Rrm$ and
$\hCp \in \Rpr$,
and its associated transfer function $\hH$ is such that
$\hH(s; \p) = \hCp \myparen{s\hEp - \hAp}^{-1} \hBp \in \Cpm$.
Therefore,
such a ROM can be simulated faster over the entire parameter domain $\pset$.
As a shorthand,
we denote systems~\eqref{eq:pLTI} and~\eqref{eq:pLTIr} by
$\Sigmap \coloneqq (E, A, B, C)$ and
$\hSigmap \coloneqq (\hE, \hA, \hB, \hC)$,
respectively.

In the simpler setting of \emph{non-parametric} $\Htwo$-optimal MOR,
a ROM's quality is measured via the Hardy $\Htwo$ norm in the $s$ variable,
i.e., the average discrepancy in the frequency domain between the FOM and ROM is
measured.
The two most well-known algorithms for $\Htwo$-optimal MOR are
the iterative rational Krylov algorithm (IRKA)~\cite{morGugAB08} and
the two-sided iteration algorithm (TSIA)~\cite{morXuZ11}.
Respectively, they are based on
the interpolatory first-order necessary conditions (FONC) from Meier and
Luenberger~\cite{morMeiL67} and
Gramian-based FONC from Wilson~\cite{morWil70}, which are respectively
called the Meier-Luenberger and Wilson conditions; for more details,
see~\cite{morGugAB08,morVanGA08,morXuZ11,morBenKS11}.

For \emph{parametric} MOR,
many approaches have been developed over the years;
see~\cite{morBenGW15,morBauBHetal17} for comprehensive surveys.
Our focus in this paper is $\HoL$-optimal parametric MOR, which was originally
proposed by Baur et al.~\cite{morBauBBetal11}.
Compared to non-parametric $\Htwo$-optimal MOR, $\HoL$-optimal MOR
additionally averages the error in the $\p$ variable
over the parameter domain.
In~\cite{morBauBBetal11}, $\HoL$-optimal parametric MOR was done under very
special assumptions, which include requiring that $E, A, \hE, \hA$ are all
constant functions of the parameter.
Another $\HoL$-optimal framework was proposed in~\cite[Chapter~$3$]{morGri18},
where Grimm modified the $\HoL$ norm considered
in~\cite{morBauBBetal11} so that FOMs such as~\eqref{eq:pLTI} are considered,
but all matrix-valued functions are assumed to be analytic
and $\hE$ and $\hA$ remain non-parametric.
Meanwhile, in~\cite[Chapter~$4$]{morPet13},
Petersson considered a discretized $\HoL$-norm optimization problem.
Here, we extend the approach of Baur et al.~\cite{morBauBBetal11} to general
parametric systems, in particular,
where all coefficients can actually vary with the parameter.

Many parametric MOR approaches build ROMs by projecting the FOM\@.
One disadvantage of this is that the FOM should
have a structure that allows the projected ROM's matrix-valued functions to be
evaluated efficiently.
If such a structure is not apparent,
then the parameter-dependent matrices of the FOM have to be approximated by, say,
empirical interpolation (see~\cite[Chapter~$2$]{morBenOCetal17}),
in order to avoid the cost of reprojecting the FOM for each new value of $\p$.
One notable approach amongst projection-based parametric MOR techniques is that
of reduced basis methods~\cite{morHesRS16,morQuaMN16}.
These aim to minimize the $\Linf$ error over the parameter domain\footnote{%
  For time-dependent problems,
  the error is additionally combined with the $\Ltwo$ error in time,
  or the $\Linf$ error in frequency if $s$ is thought of as a parameter.
}
in a greedy manner,
where the ``reduced basis'' is augmented in each step using certain solution
snapshots. To that end, an error estimator is evaluated on a
predefined training set of parameter values, and the snapshots are collected from a simulation of the FOM at the parameter value that attains the
worst (highest) estimated error. This procedure is repeated in a loop until all
estimates become sufficiently small.
In contrast, the method that we propose aims to minimize the $\HoL$ error over
the entire frequency and parameter domain by direct optimization,
where the optimization variables are the matrices defining the ROM\@.
Consequently, unlike the methods above, our approach avoids projection.

In \cref{sec:preliminaries},
we review both standard non-parametric $\Htwo$-optimal MOR
and parametric $\HoL$-optimal MOR, and then
give various definitions and results that we will need here.
Our main theoretical result,
the derivation of the gradient and Wilson-type FONC for our generalized
$\HoL$-optimal framework,
is given in~\cref{sec:fonc}.
In~\cref{sec:related_work},
we show how our new results differ with respect to the earlier results
of~\cite{morBauBBetal11,morGri18,morPet13},
while in \cref{sec:ptsia} we discuss a TSIA-like algorithm for parametric MOR
and its limitations.
In \cref{sec:opt_based},
we leverage our newly derived FONC from \cref{sec:fonc}
to propose a new optimization-based algorithm for computing locally-optimal ROMs
for $\HoL$-optimal parametric MOR\@.
Finally, we validate our new algorithm on some examples in
\cref{sec:numerical_experiments} and
make concluding remarks in \cref{sec:conclusion}.



\section{Preliminaries}\label{sec:preliminaries}
We begin with the prerequisite background.

\subsection{Non-parametric \texorpdfstring{$\Htwo$}{H2}-optimal MOR}%
Consider the non-parametric system \begin{equation}
  \label{eq:LTI}
  \begin{aligned}
    E \dot{x}(t) & = A x(t) + B u(t), \quad x(0) = 0, \\
    y(t) & = C x(t),
  \end{aligned}
\end{equation}
with the same dimensions
as~\eqref{eq:pLTI}.
We refer to~\eqref{eq:LTI} as a (non-parametric) FOM of order $n$,
and we assume that $E$ is invertible and~\eqref{eq:LTI} is asymptotically
stable.
Furthermore,
let
\begin{equation}
  \label{eq:LTIr}
  \begin{aligned}
    \hE \dot{\hx}(t) & = \hA \hx(t) + \hB u(t), \quad \hx(0) = 0, \\
    \hy(t) & = \hC \hx(t),
  \end{aligned}
\end{equation}
be a non-parametric analogue of~\eqref{eq:pLTIr}, with the same dimensions,
i.e., a ROM of order $r$.
We respectively use $\Sigma$ and $\hSigma$ to denote~\eqref{eq:LTI}
and~\eqref{eq:LTIr}.
For the state-space representation of the error system,
with transfer function $H - \hH$,
we use
\begin{equation}\label{eq:err_sys_non_para}
  \begin{aligned}
    \underbrace{%
      \begin{bmatrix}
        E & 0 \\
        0 & \hE
      \end{bmatrix}
    }_{\eqqcolon E_e}
    \begin{bmatrix}
      \dot{x}(t) \\
      \dot{\hx}(t)
    \end{bmatrix}
    & =
      \underbrace{%
        \begin{bmatrix}
          A & 0 \\
          0 & \hA
        \end{bmatrix}
      }_{\eqqcolon A_e}
      \begin{bmatrix}
        x(t) \\
        \hx(t)
      \end{bmatrix}
      +
      \underbrace{%
        \begin{bmatrix}
          B \\
          \hB
        \end{bmatrix}
      }_{\eqqcolon B_e}
      u(t), \\
    y(t) - \hy(t)
    & =
      \underbrace{%
        \begin{bmatrix}
          C & -\hC
        \end{bmatrix}
      }_{\eqqcolon C_e}
      \begin{bmatrix}
        x(t) \\
        \hx(t)
      \end{bmatrix},
  \end{aligned}
\end{equation}
where
$E_e, A_e \in \bbR^{(n + r) \times (n + r)}$,
$B_e \in \bbR^{(n + r) \times m}$, and
$C_e \in \bbR^{p \times (n + r)}$.

Following~\cite[Chapter~$5$]{morAnt05}, we introduce the Hardy $\Htwo$~norm.
Given a function $\fundef{G}{\bbC}{\Cpm}$ that is analytic in the open right
half-plane,
the $\Htwo$~norm of $G$ is
\begin{equation*}
  \normHtwo*{G}
  =
  \myparen*{\frac{1}{2\pi} \int_{-\infty}^{\infty}
  \normF*{G(\imag \omega)}^2
  \dif{\omega}}^{1/2},
\end{equation*}
where $\omega \in \bbR$ is the (angular) frequency and
$\normF{\cdot}$ denotes the Frobenius norm.
For a linear time-invariant system,
the $\Htwo$ norm is defined as the $\Htwo$ norm of its
transfer function.
Note that by our invertibility and stability assumptions,
the $\Htwo$ norm of $\Sigma$ is finite.

Even though the $\Htwo$ system norm is defined in the frequency domain,
the time-domain Lebesgue $\Linf$ norm of the output error is bounded
in terms of the $\Htwo$ norm of the error system~\eqref{eq:err_sys_non_para} and
the $\Ltwo$ norm of the input $u$ (see, e.g.,~\cite{morGugAB08}),
i.e.,
\begin{align}
  \label{eq:estimation}
  \normLinf*{y - \hy}
  \leq
  \normHtwo[\big]{H - \hH}
  \normLtwo*{u}.
\end{align}
Thus,
if we desire to find a ROM that minimizes the $\Linf$ norm of the output
error,~\eqref{eq:estimation} motivates finding the $\Htwo$-optimal ROM, i.e.,
the one that most reduces $\normHtwo{H - \hH}$.
Again, our assumptions imply that $\normHtwo{H - \hH}$ is finite.

In~\cite{morWil70},
the starting point for deriving the Wilson conditions is rewriting the squared
$\Htwo$ norm of $H - \hH$ as
\begin{equation}\label{eq:trace}
  \normHtwo[\big]{H - \hH}^2
  = \trace*{C_e P_e C_e\tran}
  = \trace*{B_e\tran Q_e B_e}.
\end{equation}
Here,
$P_e = P_e\tran$ and
$Q_e = Q_e\tran$, both in $\bbR^{(n + r) \times (n + r)}$, are determined by
the generalized Lyapunov equations
\begin{align*}
  0 & = A_e P_e E_e\tran + E_e P_e A_e\tran + B_e B_e\tran, \\
  0 & = A_e\tran Q_e E_e + E_e\tran Q_e A_e + C_e\tran C_e,
\end{align*}
and define the controllability Gramian \(P_{e}\) and the observability Gramian
\(E_e\tran Q_e E_e\) of the error system~\eqref{eq:err_sys_non_para}.
Based on~\eqref{eq:trace},
the FONC for $\Htwo$-optimality are
\begin{equation}\label{eq:Wilson}
  \begin{aligned}
    0 & = \hQ\tran \hA \hP + \tQ\tran A \tP, \\
    0 & = \hQ\tran \hE \hP + \tQ\tran E \tP, \\
    0 & = \hQ\tran \hB + \tQ\tran B, \\
    0 & = \hC \hP - C \tP,
  \end{aligned}
\end{equation}
where
$P = P\tran \in \Rnn$,
$\hP = \hP\tran \in \Rrr$, and
$\tP \in \Rnr$ are obtained from the block partitioning of the matrix
$P_e =
\begin{bsmallmatrix}
  P & \tP \\
  \tP\tran & \hP
\end{bsmallmatrix}$,
while
$Q = Q\tran \in \Rnn$,
$\hQ = \hQ\tran \in \Rrr$, and
$\tQ \in \Rnr$ are obtained from the corresponding partitioning in
$Q_e =
\begin{bsmallmatrix}
  Q & \tQ \\
  \tQ\tran & \hQ
\end{bsmallmatrix}$.

\begin{remark}\label{rem:Wilson_E_identity}
  Actually,
  in~\cite{morWil70},
  Wilson derived~\eqref{eq:Wilson} for the special case where $E = I_n$,
  the $n \times n$ identity.
  In~\cite{morVanGA08},
  it is also assumed that $E = I_n$,
  where~\eqref{eq:Wilson} is given in~\cite[Theorem~3.3]{morVanGA08}.
  As part of our derivation of new FONC for parametric systems,
  in \cref{sec:fonc} we explain how~\eqref{eq:Wilson} holds for arbitrary
  invertible matrices $E$.
\end{remark}

\subsection{Parametric \texorpdfstring{$\HoL$}{H2L2}-optimal MOR}
Now consider the parametric systems $\Sigmap$ and $\hSigmap$,
with their associated error system $H - \hH$:
\begin{equation}\label{eq:err_sys}
  \begin{aligned}
    \underbrace{%
      \begin{bmatrix}
        \Ep & 0 \\
        0 & \hEp
      \end{bmatrix}
    }_{\eqqcolon E_e(\p)}
    \begin{bmatrix}
      \dot{x}(t; \p) \\
      \dot{\hx}(t; \p)
    \end{bmatrix}
    & =
      \underbrace{%
        \begin{bmatrix}
          \Ap & 0 \\
          0 & \hAp
        \end{bmatrix}
      }_{\eqqcolon A_e(\p)}
      \begin{bmatrix}
        x(t; \p) \\
        \hx(t; \p)
      \end{bmatrix}
      +
      \underbrace{%
        \begin{bmatrix}
          \Bp \\
          \hBp
        \end{bmatrix}
      }_{\eqqcolon B_e(\p)}
      u(t), \\
    y(t; \p) - \hy(t; \p)
    & =
      \underbrace{%
        \begin{bmatrix}
          \Cp & -\hCp
        \end{bmatrix}
      }_{\eqqcolon C_e(\p)}
      \begin{bmatrix}
        x(t; \p) \\
        \hx(t; \p)
      \end{bmatrix},
  \end{aligned}
\end{equation}
where
$E_e(\p), A_e(\p) \in \bbR^{(n + r) \times (n + r)}$,
$B_e(\p) \in \bbR^{(n + r) \times m}$, and
$C_e(\p) \in \bbR^{p \times (n + r)}$.
For this parametric MOR setting,
we use the following assumptions.
\begin{assumption}[Assumptions on the FOM]\label{thm:assumptions}
  We assume:
  \begin{enumerate}[\emph{(\alph*)}]
  \item $\pset$ is a non-degenerate $d$-dimensional
  closed box aligned with the axes,\footnote{
  For simplicity and concreteness, we assume that $\pset$ is a closed box, but it actually suffices to assume
 that $\pset$ is a compact set of positive finite measure.}
  \item $\Ep$ is invertible for all $\p \in \pset$,
  \item the system $\Sigmap$ is asymptotically stable for all $\p \in \pset$,
  \item the functions $E, A, B, C$ are continuous over $\pset$.
  \end{enumerate}
\end{assumption}

Following~\cite{morBauBBetal11},
we now define parametric analogues of the $\Htwo$ and $\Linf$ norms.
For a function $\fundef{G}{\bbC \times \pset}{\Cpm}$,
such that
$G(\cdot; \p)$ is analytic in the open right half-plane
for all $\p \in \pset$,
the $\HoL$ norm of $G$ is
\begin{equation}
  \label{eq:HoL_norm}
  \normHoL{G} \coloneqq
  \myparen*{\intp \normHtwo{G(\cdot; \p)}^2 \dif{\p}}^{1/2}.
\end{equation}
Under \cref{thm:assumptions},
$\normHoL{H} < \infty$,
as $\normHtwo{H(\cdot; \p)}$ is bounded for all $\p \in \pset$ and
$\pset$ has finite measure.
The $\cL_{\infty} \otimes \Ltwo$ norm of a vector-valued function
$\fundef{v}{\bbR \times \pset}{\Rp}$~is
\begin{equation*}
  \norm{v}_{\Linf \otimes \Ltwo} \coloneqq
  \myparen*{\intp \norm{v(\cdot; \p)}_{\Linf}^2 \dif{\p}}^{1/2}.
\end{equation*}

Using~\eqref{eq:estimation} for $H(\cdot; \p)$ and $\hH(\cdot; \p)$,
it follows that the output error,
as measured by the $\Linf \otimes \Ltwo$ norm,
is bounded by the $\HoL$ norm of the error system~\eqref{eq:err_sys}, i.e.,
\begin{equation}
  \label{eq:estimation_pmor}
  \norm*{y - \hy}_{\Linf \otimes \Ltwo}
  \leq
  \normHoL[\big]{H - \hH}
  \normLtwo*{u}.
\end{equation}
If $\hEp$ is invertible and $\hSigmap$ is asymptotically stable for all
$\p \in \pset$,
then $\normHoL{H - \hH}$ is ensured to be finite.
The inequality in~\eqref{eq:estimation_pmor} immediately leads to the following
optimization problem for $\HoL$ parametric MOR
\begin{equation*}
  \min_{\hE, \hA, \hB, \hC} \quad \normHoL[\big]{H - \hH}.
\end{equation*}
Since optimizing over the space of all possible functions is clearly
impractical, in this paper
we use the following widely used parameter-separable forms for matrix-valued functions.
Specifically, we optimize over the following set of ROMs.
\begin{definition}\label{def:rom_set}
  Let $\fundef{\he_i, \ha_j, \hb_k, \hc_{\ell}}{\pset}{\bbR}$, for
  $i \in \{1, \hdots , q_{\hE}\} \eqqcolon [q_{\hE}]$,
  $j \in [q_{\hA}]$,
  $k \in [q_{\hB}]$, and
  $\ell \in [q_{\hC}]$,
  be given continuous functions.
  We define $\cR$ as the set of all ROMs whose matrix-valued functions have
  parameter-separable forms
  \begin{equation}\label{eq:rom_para_sep}
    \hEp = \sum_{i = 1}^{q_{\hE}} \he_i(\p) \hE_i,\ \
    \hAp = \sum_{j = 1}^{q_{\hA}} \ha_j(\p) \hA_j,\ \
    \hBp = \sum_{k = 1}^{q_{\hB}} \hb_k(\p) \hB_k,\ \
    \hCp = \sum_{\ell = 1}^{q_{\hC}} \hc_{\ell}(\p) \hC_{\ell},
  \end{equation}
  where
  $\hE_i, \hA_j \in \Rrr$,
  $\hB_k \in \Rrm$, and
  $\hC_{\ell} \in \Rpr$
  such that $\hEp$ is invertible and $\spab{\hAp, \hEp} < 0$ for all
  $\p \in \pset$.
\end{definition}
\begin{remark}\label{rem:remark_on_definition}
  Some comments on \cref{def:rom_set} are in order.
  First, note that \cref{def:rom_set} does not require any assumptions on the
  FOM and that $\cR$ is an open set.
  Second, we assume that the reduced-order matrix-valued
  functions~\eqref{eq:rom_para_sep} are cheap to evaluate for
  any~$\p \in \pset$.
  Typically, this means that the scalar functions
  $\he_i, \ha_j, \hb_k, \hc_{\ell}$ are both few in number and inexpensive.
  Third, while $\hEp$ being invertible is satisfied
  generically (since almost all square matrices are invertible),
  we cannot expect the additional stability assumption also to hold generically.
\end{remark}

Note that if a FOM has matrix-valued functions with forms analogous to~\eqref{eq:rom_para_sep}
and using the same scalar functions,
then a ROM $\hSigmap \in \cR$ can preserve this structure.
It can also be desirable to build ROMs with a different structure than the FOM;\
see, e.g.,~\cite{morWitTKetal16,morGosGU21}.
The method that we propose can either build ROMs that preserve the structure
of the FOM or change the given structure to a preferred one.
Even in the case that a parameter-separable form of the FOM is not given,
our approach can nevertheless design ROMs with the structure given by~\eqref{eq:rom_para_sep}.
Furthermore, we will not need any additional assumptions beyond those stated
in~\cref{thm:assumptions}.

We are now ready to present
the structured $\HoL$ optimization problem that we wish to solve
in order to obtain a parametric reduced-order model for $\Sigmap$:
\begin{equation}\label{eq:opt_problem_para}
  \min_{\hSigmap \in \cR} \quad \cJ\myparen[\big]{\hSigmap},
\end{equation}
where $\cR$ is from \cref{def:rom_set} and
$\cJ(\hSigmap) \coloneqq \normHoL{H - \hH}^2$.
We use the \emph{squared} $\HoL$ error,
as it is more convenient for deriving the respective gradients.

\begin{remark}
  Note that~\eqref{eq:opt_problem_para},
  due to the restriction of $\hSigmap \in \cR$,
  is actually a constrained optimization problem, i.e.,
  per the stability assumption in \cref{def:rom_set} defining $\cR$,
  a valid minimizer of~\eqref{eq:opt_problem_para} must also be an
  asymptotically stable system.
  In~\cref{sec:opt_based}, we explain how we satisfy this constraint
  algorithmically.
\end{remark}

Analogously to~\eqref{eq:trace} for the $\Htwo$ norm,
the $\HoL$ norm~\eqref{eq:HoL_norm} can be written as
\begin{equation}\label{eq:HoL_norm_gramian}
     \normHoL[\big]{H - \hH}^2
    = \intp \trace*{C_e(\p) P_e(\p) C_e(\p)\tran} \dif{\p}
    = \intp \trace*{B_e(\p)\tran Q_e(\p) B_e(\p)} \dif{\p},
\end{equation}
where for each $\p \in \pset$,
$P_e(\p) = P_e(\p)\tran$ and
$Q_e(\p) = Q_e(\p)\tran$, both in $\bbR^{(n + r) \times (n + r)}$,
respectively define the controllability ($P_e(\p)$) and observability ($E_e(\p)\tran Q_e(\p)
E_e(\p)$) Gramians of the error system~\eqref{eq:err_sys}.
Thus,
\begin{subequations}
  \begin{align}
    \label{eq:controllability_large}
    0
    & =
      A_e(\p) P_e(\p) E_e(\p)\tran
      + E_e(\p) P_e(\p) A_e(\p)\tran
      + B_e(\p) B_e(\p)\tran, \\
    \label{eq:observability_large}
    0
    & =
      A_e(\p)\tran Q_e(\p) E_e(\p)
      + E_e(\p)\tran Q_e(\p) A_e(\p)
      + C_e(\p)\tran C_e(\p),
  \end{align}
\end{subequations}
where
$P_e(\p) =
\begin{bsmallmatrix}
  \Pp & \tPp \\
  \tPp\tran & \hPp
\end{bsmallmatrix}$
has blocks
$\Pp =\Pp\tran \in \Rnn$,
$\hPp = \hPp\tran \in \Rrr$, and
$\tPp \in \Rnr$,
and
\begin{subequations}
  \begin{align}
    \label{eq:controllability_full}
    0 & = \Ap \Pp \Ep\tran + \Ep \Pp \Ap\tran + \Bp \Bp\tran, \\
    \label{eq:controllability_mix}
    0 & = \Ap \tPp \hEp\tran + \Ep \tPp \hAp\tran + \Bp \hBp\tran, \\
    \label{eq:controllability_red}
    0 & = \hAp \hPp \hEp\tran + \hEp \hPp \hAp\tran + \hBp \hBp\tran.
  \end{align}
\end{subequations}
Correspondingly,
$Q_e(\p) =
\begin{bsmallmatrix}
  \Qp & \tQp \\
  \tQp\tran & \hQp
\end{bsmallmatrix}$
has blocks
$\Qp =\Qp\tran \in \Rnn$,
$\hQp = \hQp\tran \in \Rrr$, and
$\tQp \in \Rnr$, and
\begin{subequations}\label{eq:observability_gramians}
  \begin{align}
    0 & = \Ap\tran \Qp \Ep + \Ep\tran \Qp \Ap + \Cp\tran \Cp, \\
    \label{eq:observability_mix}
    0 & = \Ap\tran \tQp \hEp + \Ep\tran \tQp \hAp - \Cp\tran \hCp, \\
    \label{eq:observability_red}
    0 & = \hAp\tran \hQp \hEp + \hEp\tran \hQp \hAp + \hCp\tran \hCp.
  \end{align}
\end{subequations}
Note that these blocks above, due to~\eqref{eq:HoL_norm_gramian}, will play a key role in our new FONC in \cref{sec:fonc},
and by~\cref{thm:assumptions}, these matrix-valued functions are continuous and bounded over $\pset$.

\subsection{Fr\'{e}chet differentiability}%
For our theoretical results in \cref{sec:fonc},
we need the following functional analysis definitions and results,
which follow~\cite[Chapter~4]{Zei86} and~\cite[Chapters~3,~6,~and~8]{Col12}.
Let $L(X, Y)$ denote the class of all bounded linear operators
$\fundef{A}{X}{Y}$ for normed vector spaces $X$ and $Y$.
Furthermore, let $f(x) = \lo*{\norm*{x}}$ denote that $f(x) / \norm*{x} \to 0$
as $x \to 0$.
\begin{definition}
  Let $X$ and $Y$ be normed vector spaces,
  $U \subseteq X$ open, and
  \mbox{$\fundef{f}{U}{Y}$} a function.
  The function $f$ is said to be \emph{Fr\'{e}chet differentiable} at $x \in U$
  if there exists an operator $Df(x) \in L(X, Y)$ such that
  \begin{align*}
    f(x + h) = f(x) + Df(x) h + \lo*{\norm*{h}},
  \end{align*}
  for all $h$ in some neighborhood of zero.
  The operator $Df(x)$ is called the \emph{Fr\'{e}chet derivative of $f$ at
    $x$}.
\end{definition}
\begin{definition}\label{def:derivative_inner_product}
  Let $X$ be a Hilbert space with inner product
  $\langle \cdot, \cdot \rangle_X$,
  $U \subseteq X$ open, and
  $\fundef{f}{U}{\bbR}$ a function.
  Further,
  let $f$ be Fr\'{e}chet differentiable at $x \in U$.
  The Riesz representative of $Df(x)$, i.e.,
  the element $a \in X$ such that
  \begin{align*}
    Df(x) h = \langle a, h \rangle_X,
  \end{align*}
  for all $h \in X$,
  is called the \emph{gradient of $f$ at $x$} and is denoted by $\nabla f(x)$.
\end{definition}
\begin{definition}
  Let $X$, $Y$, and $Z$ be normed vector spaces,
  $U \subset X \times Y$ open, and
  $\fundef{f}{U}{Z}$ a function.
  Let $y$ be fixed and set $g(x) = f(x, y)$.
  If $g$ is Fr\'{e}chet differentiable at $x$, then we define the
  \emph{partial Fr\'{e}chet derivative of $f$ at $(x, y)$ with respect to $x$}
  to be $D_x f(x, y) = Dg(x)$.
  The derivative $D_y f(x, y)$ and partial gradients are defined analogously.
\end{definition}



\section{First-order analysis of the squared \texorpdfstring{$\HoL$}{H2xL2}
  norm}%
\label{sec:fonc}
Our main theoretical result,
given in the following theorem,
establishes the gradient of the squared $\HoL$ norm of the error
system~\eqref{eq:err_sys} with respect to the reduced-order matrices given
by~\eqref{eq:rom_para_sep}.
This in turn directly establishes FONC for $\HoL$ optimality,
which we describe in \cref{thm:H2L2_opt} at the end of this section.
\begin{theorem}[Gradient]\label{thm:H2L2_gradients}
  Let \cref{thm:assumptions} hold.
  Furthermore,
  let $\hSigmap \in \cR$ be a structured, asymptotically stable $\HoL$-optimal
  ROM for $\Sigmap$~\eqref{eq:pLTI},
  recalling that $\hSigmap$ has matrix-valued functions in parameter-separable
  form as in~\eqref{eq:rom_para_sep}.
  Then the gradient of $\cJ(\hSigmap) = \normHoL{H - \hH}^2$ with respect to
  the fixed matrices defining~\eqref{eq:rom_para_sep} is given by
  \begin{align*}
    \nabla_{\hE_i} \cJ\myparen[\big]{\hSigmap}
    & = 2 \intp \he_i(\p)
      \myparen*{%
        \hQp\tran \hAp \hPp
        + \tQp\tran \Ap \tPp
      }
      \dif{\p},
    & i \in [q_{\hE}], \\
    \nabla_{\hA_j} \cJ\myparen[\big]{\hSigmap}
    & = 2 \intp \ha_j(\p)
      \myparen*{%
        \hQp\tran \hEp \hPp
        + \tQp\tran \Ep \tPp
      }
      \dif{\p},
    & j \in [q_{\hA}], \\
    \nabla_{\hB_k} \cJ\myparen[\big]{\hSigmap}
    & = 2 \intp \hb_k(\p)
      \myparen*{%
        \hQp\tran \hBp
        + \tQp\tran \Bp
      }
      \dif{\p},
    & k \in [q_{\hB}], \\
    \nabla_{\hC_{\ell}} \cJ\myparen[\big]{\hSigmap}
    & = 2 \intp \hc_{\ell}(\p)
      \myparen*{%
        \hCp \hPp
        - \Cp \tPp
      }
      \dif{\p},
    & \ell \in [q_{\hC}],
  \end{align*}
  with $\tP$, $\hP$ as in~\eqref{eq:controllability_mix}
  and~\eqref{eq:controllability_red} and
  $\tQ$, $\hQ$ as in~\eqref{eq:observability_mix}
  and~\eqref{eq:observability_red}.
\end{theorem}
Per \cref{rem:Wilson_E_identity},
recall that Wilson derived the analogous FONC for $\Htwo$-optimal MOR when
$E = I_n$ and $\hE = I_r$.
To prove \cref{thm:H2L2_gradients},
we first need several intermediate results.
To this end,
we now respectively generalize~\cite[Lemma~3.2 and Theorem~3.3]{morVanGA08} to
the case of general invertible $E$ and $\hE$ matrices.
\begin{lemma}\label{thm:trace_equal}
  Let
  $A, E \in \Rnn$,
  $\hA, \hE \in \Rrr$, and
  $B, C \in \Rnr$.
  If $M, N \in \Rnr$ solve the Sylvester equations
  \begin{equation}
    \label{eq:trace_equal}
    0 = A M \hE\tran + E M \hA\tran + B, \quad
    0 = A\tran N \hE + E\tran N \hA + C,
  \end{equation}
  then $\trace*{B\tran N} = \trace*{C\tran M}$.
\end{lemma}
\begin{proof}
  By substituting in $B$ from~\eqref{eq:trace_equal} into $\trace{B\tran N}$ and
  using the linearity and cyclic permutation properties of the trace,
  we have that
  \begin{align*}
    \trace*{B\tran N}
    & = \trace*{-\myparen*{A M \hE\tran + E M \hA\tran}\tran N}
      = \trace*{-\myparen*{\hE M\tran A\tran N + \hA M\tran E\tran N}} \\
    & = \trace*{-\myparen*{A\tran N \hE + E\tran N \hA} M\tran}
      = \trace*{C\tran M}.
  \end{align*}
\end{proof}
\begin{theorem}\label{thm:grad_non_para}
  Let $\Sigma$ and $\hSigma$,
  as in~\eqref{eq:LTI} and~\eqref{eq:LTIr},
  respectively with transfer functions $H$ and $\hH$,
  have invertible $E$ and $\hE$ matrices and
  be asymptotically stable non-parametric systems.
  Then,
  for $\tcJ(\hSigma) \coloneqq \normHtwo{H - \hH}^2$,
  the following hold
  \begin{subequations}
    \begin{align}
      \label{eq:Wilson_grad_E}
      \nabla_{\hE} \tcJ\myparen[\big]{\hSigma}
      & =
        2 \myparen*{\hQ\tran \hA \hP + \tQ\tran A \tP}, \\
      \nabla_{\hA} \tcJ\myparen[\big]{\hSigma}
      & =
        2 \myparen*{\hQ\tran \hE \hP + \tQ\tran E \tP}, \\
      \nabla_{\hB} \tcJ\myparen[\big]{\hSigma}
      & =
        2 \myparen*{\hQ\tran \hB + \tQ\tran B}, \\
      \nabla_{\hC} \tcJ\myparen[\big]{\hSigma}
      & =
        2 \myparen*{\hC \hP - C \tP}.
    \end{align}
  \end{subequations}
\end{theorem}
\begin{proof}
  We only derive the gradient with respect to $\hE$,
  as the other gradients follow with similar arguments.
  Given non-parametric systems~\eqref{eq:LTI} and~\eqref{eq:LTIr},
  for the duration of this proof
  we redefine~\eqref{eq:observability_gramians} as its non-parametric analogue,
  i.e., we remove the dependency on $\p$ from its equations.
  Then, via~\eqref{eq:trace},
  we have the following formulation of the squared $\Htwo$ error
  \begin{equation*}
    \tcJ\myparen[\big]{\hSigma} =
    \trace*{B\tran Q B + 2 B\tran \tQ \hB + \hB\tran \hQ \hB},
  \end{equation*}
  where matrices $Q$, $\tQ$, and $\hQ$ are from
  (again, the non-parametric analogues of)~\eqref{eq:observability_gramians}.
  Analogous to the proof of~\cite[Theorem~3.3]{morVanGA08},
  a perturbation in $\hE$ leads to perturbations in $\tQ$ and $\hQ$.
  Thus,
  by respectively replacing $\hE$ and $\tQ$
  in the Sylvester equation~\eqref{eq:observability_mix}
  with $\hE + \Delta_1$ and $\tQ + \Delta_2$,
  and then using equality~\eqref{eq:observability_mix},
  it follows that
  \begin{align}
    \notag
    0
    & =
      A\tran \myparen*{\tQ + \Delta_2} \myparen*{\hE + \Delta_1}
      + E\tran \myparen*{\tQ + \Delta_2} \hA
      - C\tran \hC \\
    \notag
    & =
      \myparen*{
      A\tran \tQ \hE
      + E\tran \tQ \hA
      - C\tran \hC}
      + A\tran \tQ \Delta_1
      + A\tran \Delta_2 \hE
      + E\tran \Delta_2 \hA
      + A\tran \Delta_2 \Delta_1 \\
    \label{eq:perturbed_sylv}
    & =
      A\tran \tQ \Delta_1
      + A\tran \Delta_2 \hE
      + E\tran \Delta_2 \hA
      + \lo*{\normF*{\Delta_1}}.
  \end{align}
  Correspondingly,
  respectively replacing $\hE$ and $\hQ$
  in the Lyapunov equation~\eqref{eq:observability_red}
  with $\hE + \Delta_1$ and $\hQ + \Delta_3$
  and then using equality~\eqref{eq:observability_red},
  we have that
  \begin{align}
    \notag
    0 = {}
    &
      \hA\tran \myparen*{\hQ + \Delta_3} \myparen*{\hE + \Delta_1}
      + \myparen*{\hE +  \Delta_1}\tran \myparen*{\hQ + \Delta_3} \hA
      + \hC\tran \hC \\
    \notag
    = {}
    &
      \myparen*{\hA\tran \hQ \hE
      + \hE \tran \hQ \hA
      + \hC\tran \hC}
      + \hA\tran \hQ \Delta_1
      + \hA\tran \Delta_3 \hE
      + \hE\tran \Delta_3 \hA
      + \Delta_1\tran \hQ \hA \\
    \notag
    &
      \quad + \hA\tran \Delta_3\Delta_1
      + \Delta_1\tran \Delta_3 \hA  \\
    \label{eq:perturbed_lyap}
    = {}
    &
      \hA\tran \hQ \Delta_1
      + \Delta_1\tran \hQ \hA
      + \hA\tran \Delta_3 \hE
      + \hE\tran \Delta_3 \hA
      + \lo*{\normF*{\Delta_1}}.
  \end{align}
  Now consider the perturbed system
  $\hSigma_{\Delta} \coloneqq (\hE + \Delta_1, \hA, \hB, \hC)$.
  This perturbation to $\hE$ results in the following perturbations
  $\hE + \Delta_1$, $\tQ + \Delta_2$, and $\hQ + \Delta_3$ in $\tcJ$,
  and so
  \begin{align}
    \notag
    \tcJ\myparen[\big]{\hSigma_{\Delta}}
    & =
      \trace*{
        B\tran Q B
        + 2 B\tran \myparen*{\tQ + \Delta_2} \hB
        + \hB\tran \myparen*{\hQ + \Delta_3} \hB
      } \\
    \label{eq:perturbation_E}
    & =
      \tcJ\myparen[\big]{\hSigma}
      + \trace*{
        2 B\tran \Delta_2 \hB
        + \hB\tran \Delta_3 \hB
      }.
  \end{align}
  In order to obtain~\eqref{eq:Wilson_grad_E},
  we first rewrite the second summand of~\eqref{eq:perturbation_E} as an inner
  product as follows.
  By applying \cref{thm:trace_equal} to
  Sylvester equations~\eqref{eq:controllability_mix} (solving for $\tP$)
  and~\eqref{eq:perturbed_sylv} (solving for $\Delta_2$)
  and then using properties of the trace,
  we have
  \begin{align}
    \label{eq:trace_equal_tE}
    2 \trace*{\hB B\tran \Delta_2}
    =
    2 \trace*{\Delta_1\tran \tQ\tran A \tP}
    + \lo*{\normF*{\Delta_1}}.
  \end{align}
  By applying the same procedure to
  Lyapunov equations~\eqref{eq:controllability_red}
  and~\eqref{eq:perturbed_lyap},
  we obtain
  \begin{align}
    \label{eq:trace_equal_hE}
    \trace*{\hB \hB\tran \Delta_3}
    =
    2 \trace*{\Delta_1\tran \hQ\tran \hA \hP}
    + \lo*{\normF*{\Delta_1}}.
  \end{align}
  Plugging~\eqref{eq:trace_equal_tE} and~\eqref{eq:trace_equal_hE}
  into~\eqref{eq:perturbation_E} yields
  \begin{align*}
    \trace*{2 B\tran \Delta_2 \hB + \hB\tran \Delta_3 \hB}
    & =
      2 \trace*{\Delta_1\tran \tQ\tran A \tP}
      + 2 \trace*{\Delta_1\tran \hQ\tran \hA \hP}
      + \lo*{\normF*{\Delta_1}} \\
    & =
      \trace*{
        2 \myparen*{\tP\tran A\tran \tQ
        + \hP\tran \hA\tran \hQ} \Delta_1
      }
      + \lo*{\normF*{\Delta_1}} \\
    & =
      \ipF*{2 \myparen*{\tQ\tran A \tP + \hQ\tran \hA \hP}}{\Delta_1}
      + \lo*{\normF*{\Delta_1}}.
  \end{align*}
  The gradient~\eqref{eq:Wilson_grad_E} follows using
  \cref{def:derivative_inner_product}.
\end{proof}
The following corollary (of \cref{thm:grad_non_para}) will be useful in proving
\cref{thm:H2L2_gradients}.
\begin{corollary}\label{thm:grad_mu}
  Let $\Sigma$ and $\hSigma$,
  as in~\eqref{eq:LTI} and~\eqref{eq:LTIr},
  respectively with transfer functions $H$ and $\hH$,
  have invertible $E$ and $\hE$ matrices and
  be asymptotically stable non-parametric systems.
  Furthermore, let the reduced-order matrices be decomposed as
  $\hE = \he_1 \hE_1 + \hE_2$,
  $\hA = \ha_1 \hA_1 + \hA_2$,
  $\hB = \hb_1 \hB_1 +\hB_2$, and
  $\hC = \hc_1 \hC_1 + \hC_2$,
  where $\he_1, \ha_1, \hb_1, \hc_1 \in \bbR$.
  Then, for $\tcJ(\hSigma) \coloneqq \normHtwo{H - \hH}^2$,
  \begin{align*}
    \nabla_{\hE_1} \tcJ\myparen[\big]{\hSigma}
    & =
      2 \he_1 \myparen*{\hQ\tran \hA \hP + \tQ\tran A \tP}, \\
    \nabla_{\hA_1} \tcJ\myparen[\big]{\hSigma}
    & =
      2 \ha_1 \myparen*{\hQ\tran \hE \hP + \tQ\tran E\tP}, \\
    \nabla_{\hB_1} \tcJ\myparen[\big]{\hSigma}
    & =
      2 \hb_1 \myparen*{\hQ\tran \hB + \tQ \tran B}, \\
    \nabla_{\hC_1} \tcJ\myparen[\big]{\hSigma}
    & =
      2 \hc_1 \myparen*{\hC \hP - C \tP},
  \end{align*}
  using the convention that $\tcJ$ is a function of
  $\hE_1$, $\hA_1$, $\hB_1$, and $\hC_1$.
\end{corollary}
As we now show,
the proof of \cref{thm:H2L2_gradients} follows from our results above.
\begin{proof}[Proof of \cref{thm:H2L2_gradients}]
  We only give the proof for the gradient with respect to $\hE_i$,
  as the others follow similarly.
  We note that $\hEp$ can be written as
  \begin{align*}
    \hEp
    & =
      \he_i(\p) \hE_i
      + \sum_{\substack{j = 1 \\ j \neq i}}^{q_{\hE}} \he_j(\p) \hE_j
     =
      \he_i(\p) \hE_i
      + \tilde{E}(\p),
  \end{align*}
  where $\tilde{E}(\p) \in \Rrr$ does not depend on $\hE_i$,
  and so we have the same structure of the matrix $\hEp$ as in
  \cref{thm:grad_mu}.
  Respectively using the Leibniz rule and \cref{thm:grad_mu} to obtain the
  second and third equalities below, we conclude the proof via
  \begin{align*}
    \nabla_{\hE_i} \cJ\myparen[\big]{\hSigmap}
    & =
      \nabla_{\hE_i}
      \intp
      \normHtwo*{H(\cdot; \p) - \hH(\cdot; \p)}^2
      \dif{\p}
    =
      \intp
      \nabla_{\hE_i}
      \normHtwo*{H(\cdot; \p) - \hH(\cdot; \p)}^2
      \dif{\p} \\
    & =
      2 \intp
      \he_i(\p)
      \myparen*{
        \hQp\tran \hAp \hPp
        + \tQp\tran \Ap \tPp
      }
      \dif{\p}.
  \end{align*}
\end{proof}
The $\HoL$ FONC,
which we first described in~\cite{morHunMS18},
is a direct consequence of \cref{thm:H2L2_gradients}.
Hence,
we restate the $\HoL$ FONC here as the following corollary.
\begin{corollary}[FONC]\label{thm:H2L2_opt}
  Let \cref{thm:assumptions} hold.
  Furthermore,
  let $\hSigmap \in \cR$ be a structured, asymptotically stable $\HoL$-optimal
  ROM for $\Sigmap$~\eqref{eq:pLTI},
  recalling that $\hSigmap$ has matrix-valued functions in parameter-separable
  form as in~\eqref{eq:rom_para_sep}.
  Then
  \begin{subequations}\label{eq:Wilson_type}
    \begin{alignat}{3}
      \label{eq:Wilson_type_E}
      0
      & =
        \intp
        \he_i(\p)
        \myparen*{
          \hQp\tran \hAp \hPp
          + \tQp\tran \Ap \tPp
        }
        \dif{\p}, \qquad
      & i \in [q_{\hE}], \\
      \label{eq:Wilson_type_A}
      0
      & =
        \intp
        \ha_j(\p)
        \myparen*{
          \hQp\tran \hEp \hPp
          + \tQp\tran \Ep \tPp
        }
        \dif{\p},
      & j \in [q_{\hA}], \\
      \label{eq:Wilson_type_B}
      0
      & =
        \intp
        \hb_k(\p)
        \myparen*{
          \hQp\tran \hBp
          + \tQp\tran \Bp
        }
        \dif{\p},
      & k \in [q_{\hB}], \\
      \label{eq:Wilson_type_C}
      0
      & =
        \intp
        \hc_{\ell}(\p)
        \myparen*{
          \hCp \hPp
          - \Cp \tPp
        }
        \dif{\p},
      & \ell \in [q_{\hC}],
    \end{alignat}
  \end{subequations}
  with $\tP, \hP$ as in~\eqref{eq:controllability_mix}
  and~\eqref{eq:controllability_red} and $\tQ, \hQ$ as
  in~\eqref{eq:observability_mix} and~\eqref{eq:observability_red}.
\end{corollary}
Due to the similarity of the equations in~\eqref{eq:Wilson_type} with the Wilson
conditions~\eqref{eq:Wilson} for non-parametric systems,
in~\cite{morHunMS18} (and later in~\cite[Theorem~6.11]{morMli20})
we referred to~\eqref{eq:Wilson_type} as
\emph{Wilson-type optimality conditions},
but for conciseness, we use FONC in this paper.



\section{Comparison to related work}\label{sec:related_work}
We now compare our results in \cref{thm:H2L2_opt} with earlier results of
Baur et al.~\cite{morBauBBetal11},
Grimm~\cite{morGri18}, and
Petersson~\cite{morPet13}.

\subsection{Parametric \texorpdfstring{$B, C$}{B, C} optimality results of Baur
  et al}%
A special case of $\HoL$-optimal parametric MOR is considered
in~\cite[Section~5.1]{morBauBBetal11},
where the FOM
\begin{align}
  \label{eq:pLTI_Baur}
  \begin{split}
    E \dot{x}(t; \p) & = A x(t; \p) + \Bp u(t), \\
    y(t; \p) & = \Cp x(t; \p),
  \end{split}
\end{align}
is a single-input, single-output system,
$E, A \in \Rnn$ are constant matrices, and
for $\p = \myparen{\p_1, \p_2}$ and $\p_1, \p_2 \in [0, 1]$,
we have the following parameter-separable forms:
\begin{equation}\label{eq:parameter_separable_FOM}
  \Bp = B_1 + \p_1 B_2 \in \bbR^{n \times 1}
  \quad \text{and} \quad
  \Cp = C_1 + \p_2 C_2 \in \bbR^{1 \times n}.
\end{equation}
The ROM is assumed to have the same structure as the FOM, i.e.,
\begin{align}
  \label{eq:pLTIr_Baur}
  \begin{split}
    \hE \dot{\hx}(t; \p) & = \hA \hx(t; \p) + \hBp u(t), \\
    \hy(t; \p) & = \hCp \hx(t; \p),
  \end{split}
\end{align}
where
$\hE, \hA \in \Rrr$ and
\begin{equation}\label{eq:parameter_separable_ROM}
    \hBp = \hB_1 + \p_1 \hB_2 \in \bbR^{r \times 1}
    \quad \text{and} \quad
    \hCp = \hC_1 + \p_2 \hC_2 \in \bbR^{1 \times r}.
\end{equation}
One of the main results of~\cite{morBauBBetal11} states that,
in this specific case,
the $\HoL$ norm of the transfer function can be expressed as a weighted $\Htwo$
norm.
\begin{theorem}[{\cite[Theorem~5.1]{morBauBBetal11}}]\label{thm:Baur}
  Let $H(s; \p) = \Cp \myparen{sE - A}^{-1} \Bp$ be the transfer function
  of~\eqref{eq:pLTI_Baur} and
  $\p = \myparen{\p_1, \p_2}$ with $\p_1, \p_2 \in [0, 1]$.
  Define the auxiliary transfer function
   $G(s) =
    \begin{bsmallmatrix}
      C_1 \\
      C_2
    \end{bsmallmatrix}
    \myparen{s E - A}^{-1}
    \begin{bsmallmatrix}
      B_1 & B_2
    \end{bsmallmatrix}$
    and
    $
    L =
    \begin{bsmallmatrix}
    	1 & 0 \\
    	\frac{1}{2} & \frac{1}{2\sqrt{3}}
     \end{bsmallmatrix}$.
  Then $\normHoL{H} = \normHtwo{L\tran G L}$.
\end{theorem}

To show how our FONC~\eqref{eq:Wilson_type} generalize
results from~\cite{morBauBBetal11}, first note that \cref{thm:Baur} also
implies that
$\normHoL[\big]{H - \hH} = \normHtwo[\big]{L\tran G L - L\tran \hG L}$
holds with transfer function
$\hG(s) =
\begin{bsmallmatrix}
  \hC_1 \\
  \hC_2
\end{bsmallmatrix}
\myparen{s \hE- \hA}^{-1}
\begin{bsmallmatrix}
  \hB_1 & \hB_2
\end{bsmallmatrix}$.
Now letting
\[
  B_L =
  \begin{bsmallmatrix}
    B_1 & B_2
  \end{bsmallmatrix}L,
  \quad
  \hB_L =
  \begin{bsmallmatrix}
    \hB_1 & \hB_2
  \end{bsmallmatrix}L,
  \quad
  C_L =
  L\tran
  \begin{bsmallmatrix}
    C_1 \\ C_2
  \end{bsmallmatrix},
  \quad \text{and} \quad
  \hC_L =
  L\tran
  \begin{bsmallmatrix}
    \hC_1 \\ \hC_2
  \end{bsmallmatrix},
\]
and observing that $L\tran G L$ and $L\tran \hG L$ are, respectively,
the transfer functions of $\myparen{E, A, B_L, C_L}$ and
$\myparen{\hE, \hA, \hB_L, \hC_L}$,
it follows that the FONC for this setting are
\begin{equation}\label{eq:Wilson_Baur}
  \begin{aligned}
    0
    & =
      \hQ_L\tran \hE \hP_L
      + \tQ_L\tran E \tP_L, \\
    0
    & =
      \hQ_L\tran \hA \hP_L
      + \tQ_L\tran A \tP_L, \\
    0
    & =
      \hQ_L\tran \hB_L
      + \tQ_L\tran B_L, \\
    0
    & =
      \hC_L \hP_L
      - C_L \tP_L,
  \end{aligned}
\end{equation}
with $\hP_L$, $\tP_L$, $\hQ_L$, $\tQ_L$ such that
\begin{equation}\label{eq:baur-pql}
  \begin{alignedat}{4}
    0
    & =
      \hA \hP_L \hE\tran
      + \hE \hP_L \hA\tran
      + \hB_L \hB_L\tran, \qquad
    &
    0
    & =
      \hA\tran \hQ_L \hE
      + \hE\tran \hQ_L \hA
      + \hC_L\tran \hC_L, \\
    0
    & =
      A \tP_L \hE\tran
      + E \tP_L \hA\tran
      + B_L \hB_L\tran, \qquad
    &
    0
    & =
      A\tran \tQ_L \hE
      + E\tran \tQ_L \hA
      -  C_L\tran \hC_L.
  \end{alignedat}
\end{equation}
We now elaborate how the FONC~\eqref{eq:Wilson_Baur} for
the parametric MOR problem defined by
\cref{eq:pLTI_Baur,%
  eq:parameter_separable_FOM,%
  eq:pLTIr_Baur,%
  eq:parameter_separable_ROM}
  is actually a special case of our more general FONC~\eqref{eq:Wilson_type}.
\begin{lemma}
  Let $\Sigmap$ and $\hSigmap$, respectively, be asymptotically stable systems as
  in~\eqref{eq:pLTI_Baur} and~\eqref{eq:pLTIr_Baur},
  $E$ and $\hE$ invertible,
  $B$, $C$, $\hB$, and $\hC$ as in~\eqref{eq:parameter_separable_FOM}
  and~\eqref{eq:parameter_separable_ROM},
  and $\p = \myparen{\p_1, \p_2}$ with $\p_1, \p_2 \in [0, 1]$.
  Then $\tP$, $\hP$, $\tQ$, and $\hQ$ satisfy
  \begin{subequations}\label{eq:form_thPQ}
    \begin{align}
      \label{eq:form_tP}
      \tPp & = \tP_1 + \p_1 \tP_2 + \p_1^2 \tP_3, \\
      \label{eq:form_hP}
      \hPp & = \hP_1 + \p_1 \hP_2 + \p_1^2 \hP_3, \\
      \label{eq:form_tQ}
      \tQp & = \tQ_1 + \p_2 \tQ_2 + \p_2^2 \tQ_3, \\
      \label{eq:form_hQ}
      \hQp & = \hQ_1 + \p_2 \hQ_2 + \p_2^2 \hQ_3,
    \end{align}
  \end{subequations}
  where for $i = 1, 2, 3$, matrices
  $\tP_i, \tQ_i \in \Rnr$ are tall and skinny with
  \begin{subequations}\label{eq:form_part_tP}
    \begin{align}
      \label{eq:form_part_tP1}
      0
      & =
        A \tP_1 \hE\tran
        + E \tP_1 \hA\tran
        + B_1 \hB_1\tran, \\
      \label{eq:form_part_tP2}
      0
      & =
        A \tP_2 \hE\tran
        + E \tP_2 \hA\tran
        + B_1 \hB_2\tran
        + B_2\hB_1\tran, \\
      \label{eq:form_part_tP3}
      0
      & =
        A \tP_3 \hE\tran
        + E \tP_3 \hA\tran
        + B_2 \hB_2\tran, \\[1ex]
      \notag
      0
      & =
        A\tran \tQ_1 \hE
        + E\tran \tQ_1 \hA
        - C_1\tran \hC_1, \\
      \notag
      0
      & =
        A\tran \tQ_2 \hE
        + E\tran \tQ_2 \hA
        - C_1\tran \hC_2
        - C_2\tran\hC_1, \\
      \notag
      0
      & =
        A\tran \tQ_3 \hE
        + E\tran \tQ_3 \hA
        - C_2\tran \hC_2,
    \end{align}
  \end{subequations}
  while $\hP_i, \hQ_i \in \Rrr$ are small square matrices and
  \begin{alignat*}{4}
    0
    & =
      \hA \hP_1 \hE\tran
      + \hE \hP_1 \hA\tran
      + \hB_1 \hB_1\tran, \quad
    &
    0
    & =
      \hA\tran \hQ_1 \hE
      + \hE\tran \hQ_1 \hA
      + \hC_1\tran \hC_1, \\
    0
    & =
      \hA \hP_2 \hE\tran
      + \hE \hP_2 \hA\tran
      + \hB_1 \hB_2\tran
      + \hB_2\hB_1\tran, \quad
    &
    0
    & =
      \hA\tran \hQ_2 \hE
      + \hE\tran \hQ_2 \hA
      + \hC_1\tran \hC_2
      + \hC_2\tran\hC_1, \\
    0
    & =
      \hA \hP_3 \hE\tran
      + \hE \hP_3 \hA\tran
      + \hB_2 \hB_2\tran, \quad
    &
    0
    & =
      \hA\tran \hQ_3 \hE
      + \hE\tran \hQ_3 \hA
      + \hC_2\tran \hC_2.
  \end{alignat*}
\end{lemma}
\begin{proof}
  Substituting both $\Bp$ from~\eqref{eq:parameter_separable_FOM} and $\hBp$
  from~\eqref{eq:parameter_separable_ROM} into~\eqref{eq:controllability_mix},
  we obtain that $\tPp$ is a solution of the following Sylvester equation
  \begin{align}
    \label{eq:Sylvester_decomposed}
    0 = {}
    &
      A X \hE\tran
      + E X \hA\tran
      + B_1 \hB_1\tran
      + \p_1 \myparen*{B_1 \hB_2\tran + B_2 \hB_1\tran}
      + \p_1^2 B_2 \hB_2\tran,
  \end{align}
  where $X \in \bbR^{n \times n}$ is unknown.
  Meanwhile, by linearly
  combining~\eqref{eq:form_part_tP1}--\eqref{eq:form_part_tP3},
  we see that $\tP_1 + \p_1 \tP_2 + \p_1^2 \tP_3$ also
  solves~\eqref{eq:Sylvester_decomposed}.
  We now explain that~\eqref{eq:Sylvester_decomposed} in fact has a unique
  solution,
  and so $\tPp = \tP_1 + \p_1 \tP_2 + \p_1^2 \tP_3$.
  For any given $\p \in [0, 1] \times [0, 1]$,
  by~\cite[Theorem~1]{Chu87},
  there exists a unique solution to~\eqref{eq:Sylvester_decomposed}
  if the matrix pencils $A - \lambda E$ and $-\hA\tran - \lambda \hE\tran$
  are both regular and have no eigenvalues in common.
  Since we assume that $E$ and $\hE$ are nonsingular,
  both pencils must be regular.
  Furthermore, as we also assume that $\alpha(A, E) < 0$  and
  $\alpha(\hA, \hE) < 0$,
  the two pencils cannot share eigenvalues since $\alpha(\hA, \hE) < 0$ means
  that all eigenvalues of $-\hA\tran - \lambda \hE\tran$ are in the open right
  half-plane.
  The equalities in \cref{eq:form_hP,eq:form_tQ,eq:form_hQ} are obtained in an
  analogous fashion.
\end{proof}

\begin{theorem}
  Let $\Sigmap$ and $\hSigmap$, respectively, be asymptotically stable systems as
  in~\eqref{eq:pLTI_Baur} and~\eqref{eq:pLTIr_Baur},
  $E$ and $\hE$ invertible,
  $B$, $C$, $\hB$, and $\hC$ as in~\eqref{eq:parameter_separable_FOM}
  and~\eqref{eq:parameter_separable_ROM}.
  Then, \eqref{eq:Wilson_Baur} is a special case of FONC~\eqref{eq:Wilson_type}.
\end{theorem}
\begin{proof}
  Substituting the parameter-separable forms~\eqref{eq:form_thPQ}
  into~\eqref{eq:Wilson_type_E},
  we obtain
  \begin{align*}
    0 = {}
    &
      \intp
      \myparen*{
        \hQp\tran \hA \hPp
        + \tQp\tran A \tPp
      }
      \dif{\p} \\
    = {}
    &
      \intz \intz
      \left(
        \myparen*{\hQ_1 + \p_2 \hQ_2 + \p_2^2 \hQ_3}\tran
        \hA
        \myparen*{\hP_1 + \p_1 \hP_2 + \p_1^2 \hP_3}
      \right. \\
    & \qquad\qquad +
      \left.
        \myparen*{\tQ_1 + \p_2 \tQ_2 + \p_2^2 \tQ_3}\tran
        A
        \myparen*{\tP_1 + \p_1 \tP_2 + \p_1^2 \tP_3}
      \right)
      \dif{\p_1}
      \dif{\p_2} \\
    = {}
    &
      \hW\tran \hA \hV
      + \tW\tran A \tV,
  \end{align*}
  where
  \begin{equation}\label{eq:baur-vw}
    \begin{alignedat}{4}
      \hV & \coloneqq \hP_1 + \tfrac{1}{2} \hP_2 + \tfrac{1}{3} \hP_3, & \qquad
      \tV & \coloneqq \tP_1 + \tfrac{1}{2} \tP_2 + \tfrac{1}{3} \tP_3, \\
      \hW & \coloneqq \hQ_1 + \tfrac{1}{2} \hQ_2 + \tfrac{1}{3} \hQ_3, &
      \tW & \coloneqq \tQ_1 + \tfrac{1}{2} \tQ_2 + \tfrac{1}{3} \tQ_3.
    \end{alignedat}
  \end{equation}
  Following the same procedure for~\eqref{eq:Wilson_type_A},
  we have that
  \begin{equation*}
    0 = \hW\tran \hE \hV + \tW\tran E \tV.
  \end{equation*}
  Noting that
  $LL\tran =
  \begin{bsmallmatrix}
    1 & \sfrac{1}{2} \\
    \sfrac{1}{2} & \sfrac{1}{3}
  \end{bsmallmatrix}$,
  from~\eqref{eq:Wilson_type_B} it follows that
  \begin{align*}
    \begingroup 
    \setlength\arraycolsep{2.25pt}
    \begin{bmatrix}
      0 & 0
    \end{bmatrix}
    \endgroup
    \! = {}
    &
      \intz \intz
      \begingroup 
      \setlength\arraycolsep{2.25pt}
      \begin{bmatrix}
        1 & \p_1
      \end{bmatrix}
      \endgroup
           \left(
        \myparen*{\hQ_1 + \p_2 \hQ_2 + \p_2^2 \hQ_3}\tran
        \myparen*{\hB_1 + \p_1 \hB_2}
      \right. \\
    & \qquad\qquad\qquad\qquad +
      \left.
        \myparen*{\tQ_1 + \p_2 \tQ_2 + \p_2^2 \tQ_3}\tran
        \myparen*{B_1 + \p_1 B_2}
      \right)
      \dif{\p_1}
      \dif{\p_2} \\
    = {}
    &
      \mybrack*{
        \begin{array}{@{}c|c@{}}
          \hW\tran \myparen*{\hB_1 + \tfrac{1}{2} \hB_2}
          + \tW\tran \myparen*{B_1 + \tfrac{1}{2} B_2}
          &
            \hW\tran \myparen*{\tfrac{1}{2} \hB_1 + \tfrac{1}{3} \hB_2}
            + \tW\tran \myparen*{\tfrac{1}{2} B_1 + \tfrac{1}{3} B_2}
        \end{array}
      } \\
    = {}
    &
      \hW\tran
      \begin{bmatrix}
        \hB_1 & \hB_2
      \end{bmatrix}
      L L\tran
      +
      \tW\tran
      \begin{bmatrix}
        B_1 & B_2
      \end{bmatrix}
      L L\tran.
  \end{align*}
  Multiplying the equation above on the right by the inverse of $L\tran$,
  we obtain
  \begin{equation*}
    0 =
    \hW\tran \hB_L
    +
    \tW\tran B_L.
  \end{equation*}
  Analogously,~\eqref{eq:Wilson_type_C} simplifies to
  \begin{equation*}
    0 =
    \hC_L \hV
    -
    C_L \tV.
  \end{equation*}
  It remains to show that $\hP_L$, $\tP_L$, $\hQ_L$, $\tQ_L$
  in~\eqref{eq:baur-pql} are equal, respectively,
  to $\hV$, $\tV$, $\hW$, $\tW$ in~\eqref{eq:baur-vw}.
  We show this only for $\tV$,
  as it follows similarly for the remaining matrices.
  By the following weighted sum
  $\eqref{eq:form_part_tP1}
  + \tfrac{1}{2}\eqref{eq:form_part_tP2}
  + \tfrac{1}{3}\eqref{eq:form_part_tP3}$,
  we obtain
  \begin{align*}
    0 = {}
    &
      A \tV \hE\tran
      + E \tV \hA\tran
      + B_1 \hB_1\tran
      + \tfrac{1}{2} \myparen*{B_1 \hB_2\tran + B_2 \hB_1\tran}
      + \tfrac{1}{3} B_2 \hB_2\tran \\
    = {}
    &
      A \tV \hE\tran
      + E \tV \hA\tran
      +
      \begin{bmatrix}
        B_1 & B_2
      \end{bmatrix}
      L L\tran
      \begin{bmatrix}
        \hB_1 & \hB_2
      \end{bmatrix}\tran \\
      = {}
      &
       A \tV \hE\tran
      + E \tV \hA\tran
      + B_L \hB_L\tran.
  \end{align*}
\end{proof}

\subsection{Unit-disk optimality results of Grimm}%
Let $\bbC_+$ denote the open right half-plane and
$\mathbb{D} = \{z \in \bbC \colon \abs{z} \le 1\}$ the complex unit disk.
Grimm~\cite[Chapter~3]{morGri18} considered the following variant of the $\HoL$
norm for single-input, single-output systems:
\begin{equation*}
  \normHoL*{G} \coloneqq
  \myparen*{
    \frac{1}{4 \pi^2}
    \int_{-\infty}^{\infty}
    \int_{0}^{2 \pi}
    \abs*{G\myparen*{\imag \omega, e^{\imag\theta}}}^2
    \dif{\theta} \dif{\omega}
  }^{1 / 2},
\end{equation*}
where function $\fundef{G}{\bbC \times \bbC}{\bbC}$ is analytic on
$\bbC_+ \times \mathbb{D}$.
The FOM is assumed to have the form as in~\eqref{eq:pLTI}, while
Grimm additionally assumes that $E, A, B, C$ are analytic.
The aim is to find an optimal ROM whose transfer function is of the form
\begin{equation*}
  \hH(s; \p) =
  \sum_{i = 1}^{r_s}
  \sum_{j = 1}^{r_p}
  \frac{\phi_{i, j}}{(s - \lambda_i) (\p - \pi_j)},
\end{equation*}
where
$\phi_{i, j} \in \bbC$,
$\lambda_i \in \bbC$ with $\Real{\lambda_i} < 0$, and
$\pi_j \in \bbC$ with $\abs{\pi_j} > 0$,
for $i \in [r_s]$ and $j \in [r_p]$,
with some additional assumptions on $\phi_{i, j}$, $\lambda_i$ and $\pi_j$ such
that the transfer function $\hH$ is real
(see~\cite[Section~$3.2.2$ and Lemma~$3.2.9$]{morGri18}).
The interpolatory FONC were derived in~\cite[Theorem~3.3.4]{morGri18} and an
optimization-based approach for satisfying them was proposed.

Note that one possible state-space realization of $\hH$ is
\begin{align*}
  \dot{\hx}(t; \p)
  & =
    \begin{bsmallmatrix}
      \lambda_1 \\
      & \ddots \\
      & & \lambda_{r_s}
    \end{bsmallmatrix}
    \hx(t; \p) +
    \begin{bsmallmatrix}
      1 \\
      \vdots \\
      1
    \end{bsmallmatrix}
    u(t), \\
  \hy(t; \p)
  & =
    \sum_{j = 1}^{r_p}
    \frac{1}{\p - \pi_j}
    \begin{bmatrix*}
      \phi_{1, j}
      & \cdots
      & \phi_{r_s, j}
    \end{bmatrix*}
    \hx(t; \p).
\end{align*}
Defining the matrix $\Phi = {[\phi_{i,j}]}_{i \in [r_s], j \in [r_p]}$,
$\hC$ can be written in parameter-separable form
$
  \hCp = \sum_{j=1}^{r_p} \frac{1}{\p - \pi_j} \myparen{\Phi e_j}\tran,
$
with functions $\hc_{j}(\p) = \frac{1}{\p - \pi_j}$ and reduced-order matrices
$\hC_j = \myparen{\Phi e_j}\tran$ for $j \in [r_p]$.
The main similarity to our setting is the usage of parameter-separable forms.
However, there are appreciable differences.
The first is that here the scalar parameter is varying in $\mathbb{D}$,
while in our setting, the parameters live in a subset of $\Rd$.
The second is that only $\hC$ varies with the parameter, and so, the poles
$\lambda_i$ of the ROM are fixed.
However, when either $E$ or $A$ is parametric, the system poles of the FOM
vary with respect to $\p$, and indeed the poles may be quite sensitive to
changes in $\p$.
Consequently, using a ROM with fixed poles may significantly limit its modeling
performance over the entire parameter domain.
In fact, in~\cref{sec:example_triple_chain}, we provide an experiment
illustrating exactly this.
Finally, here the scalar functions $\hc_j$ are not fixed in advance and the
scalars $\pi_j$ are optimized,
but our approach could be directly extended to additionally allow optimization
variables in the scalar functions.

\subsection{Discretized optimality results of Petersson}%
A discretized, frequency-weighted $\HoL$-like objective function
is considered in~\cite[Chapter~$5$]{morPet13}, i.e.,
\begin{align}\label{eq:petersson-obj-fun}
  \cJ_{\mathrm{d}}(\hSigmap)
  \coloneqq
  \sum_{i = 1}^{n_p}
  \normHtwo*{
    W_i^{\text{out}}
    \myparen*{H\myparen[\big]{\cdot; \p^{(i)}}
      - \hH\myparen[\big]{\cdot; \p^{(i)}}}
    W_i^{\text{in}}
  }^2,
\end{align}
for fixed $\p^{(1)}, \p^{(2)}, \dots, \p^{(n_p)} \in \pset$,
where $W_i^{\text{in}}$ and $W_i^{\text{out}}$ are asymptotically stable
transfer functions of input and output weights with respect to the parameter
$\p^{(i)}$ for $i \in [n_p]$.
Petersson uses similar parameter-separable forms for $\hA, \hB, \hC$ as we do,
except with fixed $\hE(\p) = I$.
For the function in~\eqref{eq:petersson-obj-fun},
Petersson derived Wilson-type FONC and also proposed an optimization-based approach
to satisfy them.

To compare this approach to ours,
note that the Lebesgue measure in the definition $\HoL$ norm~\eqref{eq:HoL_norm}
can be replaced with any nontrivial, finite, Borel measure~$\measure$ over
$\pset$, and our results will still hold.
In fact,
the only necessary change is using the measure theory version of the Leibniz
rule~\cite[Theorem~2.27]{Fol17} in the proof of \cref{thm:H2L2_gradients}.
Now letting the measure~$\measure$ be the sum of Dirac measures
$\delta_{\p^{(i)}}$,
i.e., $\mu(S) = \card{S \cap \{\p^{(1)}, \p^{(2)}, \dots, \p^{(n_p)}\}}$ for
every subset $S$ of $\pset$,
we recover the non-weighted discretized $\HoL$ norm as
in~\eqref{eq:petersson-obj-fun}.
Therefore, after including weights, it can be seen that our work is a
generalization of Petersson’s results.
Furthermore,
for the setting where $\measure$ is a Lebesgue measure,
the disadvantage of using~\eqref{eq:petersson-obj-fun} is that the parameter
values $\p^{(1)}, \p^{(2)}, \dots, \p^{(n_p)}$ have to be chosen in advance,
i.e., fixed.
In contrast, our approach allows for adaptive quadrature and thus provides
guarantees on the numerical accuracy, unlike parameter values chosen a priori.



\section{A TSIA-based algorithm for parametric MOR}%
\label{sec:ptsia}
For non-parametric systems,
TSIA is the fixed-point iteration given by
\begin{alignat*}{4}
  \hE_{k + 1} & = -\hQ_k\mtran \tQ_k\tran E \tP_k \hP_k^{-1}, \quad &
  \hA_{k + 1} & = -\hQ_k\mtran \tQ_k\tran A \tP_k \hP_k^{-1}, \\
  \hB_{k + 1} & = -\hQ_k\mtran \tQ_k\tran B, \quad &
  \hC_{k + 1} & = C \tP_k \hP_k^{-1},
\end{alignat*}
based on the Wilson conditions~\eqref{eq:Wilson},
where $\tP_k$, $\tQ_k$, $\hP_k$, $\hQ_k$ are Gramian blocks coming from
$(\hE_k, \hA_k, \hB_k, \hC_k)$.
Furthermore,
the inverses of $\hP_k$ and $\hQ_k$ are in fact not even needed in TSIA,
since
\begin{equation*}
  \hE_{k + 1} = \tQ_k\tran E \tP_k, \quad
  \hA_{k + 1} = \tQ_k\tran A \tP_k, \quad
  \hB_{k + 1} = \tQ_k\tran B, \quad
  \hC_{k + 1} = C \tP_k,
\end{equation*}
defines an equivalent ROM
(also, $\tP_k$ and $\tQ_k$ can be first orthonormalized).

When we began this work on $\HoL$ MOR, we originally considered
extending TSIA using the parametric FONC~\eqref{eq:Wilson_type}.
However, limitations with this approach quickly became apparent.
First, this does not permit all parameter-separable forms in the ROM\@.
For instance,
inserting $\hA$ from~\eqref{eq:rom_para_sep} into~\eqref{eq:Wilson_type_E}
yields
\begin{align*}
  0
  & =
    \intp
    \he_i(\p)
    \myparen*{
      \hQp\tran \sum_{j = 1}^{q_{\hA}} \myparen*{\ha_j(\p) \hA_j} \hPp
      + \tQp\tran \Ap \tPp
    }
    \dif{\p},
  & i \in [q_{\hE}],
\end{align*}
which is a linear system of $q_{\hE} r^2$ equations and $q_{\hA} r^2$ unknowns
(the entries of the matrices $\hA_j$).
Therefore, to have a unique solution, $q_{\hA}$ and $q_{\hE}$ must be equal.
Second,
even for the simple example
\begin{equation*}
  \begin{alignedat}{4}
    q_{\hA} &= q_{\hE} = 2, \quad \qquad &
    \ha_1(\p) &= \he_1(\p) = \hb_1(\p) = \hc_1(\p) = 1, \\
    q_{\hB} &= q_{\hC} = 1, &
    \ha_2(\p) &= \he_2(\p) = \p,
  \end{alignedat}
\end{equation*}
we observed that our parametric TSIA algorithm would not converge in practice,
and in particular, would not preserve stability.
A key difference in this parametric setting is that
it does not seem possible to eliminate $\hP_k(\p)$ and $\hQ_k(\p)$ (unlike in
non-parametric TSIA).
This can be problematic since inverses of matrices involving $\hP_k$ and $\hQ_k$
are needed to obtain the next iterate, and
these matrices may be ill-conditioned for $\p \in \pset$.
In contrast, our new approach, which we are about to describe,
can use any parameter-separable form for the ROM without ever needing such
inverses.



\section{An optimization-based approach for parametric MOR}%
\label{sec:opt_based}
The objective function in~\eqref{eq:opt_problem_para} is smooth,
and via \cref{thm:H2L2_gradients},
we have derived its gradient with respect to the reduced-order matrices given
in~\eqref{eq:rom_para_sep}.
Before we take advantage of these properties for our new optimization-based
algorithm for $\HoL$-optimal MOR,
we first address the stability constraint in~\eqref{eq:opt_problem_para}.

\subsection{Evaluating the stability constraint}
For a given $\p \in \pset$, it is easy to check
whether or not $\spab{\hAp, \hEp} < 0$ holds,
but checking whether a ROM is asymptotically stable over all $\p \in \pset$ is more difficult.
We propose testing whether $\max_{\p \in \pset} \spab{\hAp, \hEp} < 0$ holds via Chebfun~\cite{DriHT14},
which ``is an open-source package for computing with functions to about
$15$-digit accuracy.''\footnote{This quote is taken from
  \url{http://www.chebfun.org}, where Chebfun can be downloaded.}
By having Chebfun build a high-fidelity interpolant approximation (called a \texttt{chebfun}) to
$\spab{\hAp, \hEp}$ on domain~$\pset$,
we can then easily and efficiently obtain its global maximizer(s).
This allows us to reliably ascertain stability over $\pset$.

\subsection{Our optimization-based MOR algorithm}
Since the objective function in~\eqref{eq:opt_problem_para} is smooth
and we can compute its gradient,
we can consider applying fast optimization techniques, e.g., BFGS,
in order to compute locally optimal ROMs.
The number of optimization variables for a ROM given
by~\eqref{eq:rom_para_sep} is
$N = \myparen{q_{\hE} + q_{\hA}} r^{2} + \myparen{q_{\hB} m + q_{\hC} p} r$.
In many settings, $N$ will be relatively small, e.g., since $r \ll n$,
which makes BFGS an appropriate and efficient choice;
BFGS does $\bo{N^2}$ work per iteration and uses $\bo{N^2}$ memory
but converges superlinearly under sufficient smoothness conditions.
If $N$ is large enough to make BFGS impractical, limited-memory BFGS
(L-BFGS) is a good alternative.
Thus, we expect the cost of such an optimization-based algorithm actually to be
dominated by the costs just to evaluate the objective function~$\cJ$ and its
gradient at different ROMs encountered by the algorithm.
While directly evaluating $\cJ$ and $\nabla \cJ$ would be expensive,
we now explain how these computations can be made much cheaper.

Note that the objective function given in~\eqref{eq:opt_problem_para} can
be written as
\begin{align}
  \notag
  \cJ\myparen[\big]{\hSigmap}
  = {}
  &
    \intp
    \trace*{
      \Cp \Pp \Cp\tran
      + \hCp \hPp \hCp\tran
      - 2 \Cp \tPp \hCp\tran
    }
    \dif{\p} \\
  \label{eq:h2l2_sep}
  = {}
  &
    \normHoL{H}^2
    +
    \intp
    \trace*{
      \hCp \hPp \hCp\tran
      - 2 \Cp \tPp \hCp\tran
    }
    \dif{\p}
\end{align}
Since the first term in~\eqref{eq:h2l2_sep}, i.e.,
the squared $\HoL$ norm of the FOM,
is actually a constant,
for the purpose of finding minimizers,
we can replace $\cJ$ in~\eqref{eq:opt_problem_para} with
\begin{align}
  \label{eq:better_J}
  \bJ\myparen[\big]{\hSigmap}
  =
  \intp
  \trace*{
    \hCp \hPp \hCp\tran
    - 2 \Cp \tPp\hCp\tran
  }
  \dif{\p},
\end{align}
noting that $\nabla \bJ = \nabla \cJ$ of course still holds.
For a given $\p \in \pset$,
$\tPp$ in~\eqref{eq:better_J} is obtained by solving the $n \times r$ Sylvester
equation~\eqref{eq:controllability_mix}, which in turn involves solving
shifted linear systems $(s \Ep - \Ap) x = b$.
Using the convention that solving a (shifted) linear system is $\bo{S}$ work,
where $S$ varies between $n$ and $n^3$ depending on the sparsity and the
structure of the matrix and the solver that is used, it follows that
$\tPp$ can be obtained in $\bo{r S + n r^2}$ work~\cite{morBenKS11}.
Meanwhile, via solving the $r \times r$ Lyapunov
equation~\eqref{eq:controllability_red},
we obtain $\hPp$ in~\eqref{eq:better_J} in just $\bo{r^3}$ work~\cite{BarS72}.
In contrast, if we were to evaluate~\eqref{eq:h2l2_sep},
we would need $\Pp$,
which would involve solving a large $n \times n$ Lyapunov
equation~\eqref{eq:controllability_full}, and
thus be much more expensive to compute.
While~\eqref{eq:better_J} and its gradient requires evaluating an integral for
each entry in the ROM matrices in~\eqref{eq:rom_para_sep},
these integrals can be bundled together in a single call to \texttt{integral} in
\matlab{} such that the integrals make use of the same quadrature points.
This is a crucial implementation detail as otherwise $\tPp$ and $\hPp$ would
likely be recomputed many times over for the same values of $\p \in \pset$
during integration.
Finally, note that we do not even evaluate~\eqref{eq:better_J} (and its
gradient) if we detect that $\max_{\p \in \pset} \spab{\hAp, \hEp} \geq 0$;
in this case, we simply return $\bJ=\infty$ and any vector for~$\nabla \bJ$.
This serves two purposes.
First, it ensures that the optimization solver will only accept asymptotically
stable ROMs on each iteration,
thus guaranteeing we always compute feasible (asymptotically stable) solutions
to~\eqref{eq:h2l2_sep}.
This also means that we do not need to use a solver for constrained optimization to find solutions to~\eqref{eq:h2l2_sep},
i.e., an unconstrained optimization solver suffices.
Second, we also avoid the cost of evaluating $\bJ$ and $\nabla \bJ$ at any unstable
ROMs encountered, e.g., within line searches.

\subsection{Additional implementation and termination details}
Although our optimization-based algorithm ensures that it will only ever accept
asymptotically stable ROMs on each iteration, it still must be initialized at
an asymptotically stable ROM\@.
But this is generally easy to satisfy via a variety of techniques.
For example,
assuming that $\he_1$ and $\ha_1$ are positive over $\pset$,
we can easily construct an invertible $\hE_1$ and $\hA_1$ such
that $\spab{\hA_1, \hE_1} < 0$ and then set $\hE_i = 0$ and $\hA_j = 0$
for \mbox{$i, j \ge 2$}.
Alternatively, starting with an arbitrary ROM with matrix-valued
functions of the form as in~\eqref{eq:rom_para_sep} but not in $\cR$,
one can use nonsmooth optimization techniques to minimize
$\max_{\p \in \pset} \spab[\big]{\hA(\p), \hE(\p)}$ until it is negative and
then use the resulting ROM to initialize our method.
Even though our method ensures asymptotically stable ROMs at every iteration,
encountering nearly unstable ROMs may cause numerical
issues when computing $\tPp$ and $\hPp$, e.g.,
these computed matrices may contain \texttt{inf} or \texttt{NaN} entries
or \texttt{lyap} may throw an error when solving~\eqref{eq:controllability_red}.
However, these scenarios are trivially handled by detecting these cases
and once again just returning $\infty$ for the value of $\bJ$.
Finally, we propose the following tolerance criterion to determine
when to halt optimization:
\begin{align}
  \label{eq:our_conv_crit}
  \frac{\normHoL[\big]{\hH^i - \hH^{i-1}}}{\normHoL[\big]{\hH^{i-1}}}
  < \textsf{tol},
\end{align}
where $\hH^{i}$ denotes the transfer function of a ROM~\eqref{eq:pLTIr} in the
$i$th iteration of optimization.
Note that~\eqref{eq:our_conv_crit} is cheap to compute (it uses only
ROM matrices) and measures the relative change in performance between consecutive ROMs during
optimization.
Thus, when the ROMs are not changing significantly, this quantity will be small.



\section{Numerical experiments}\label{sec:numerical_experiments}
All
\begin{figure}[t]
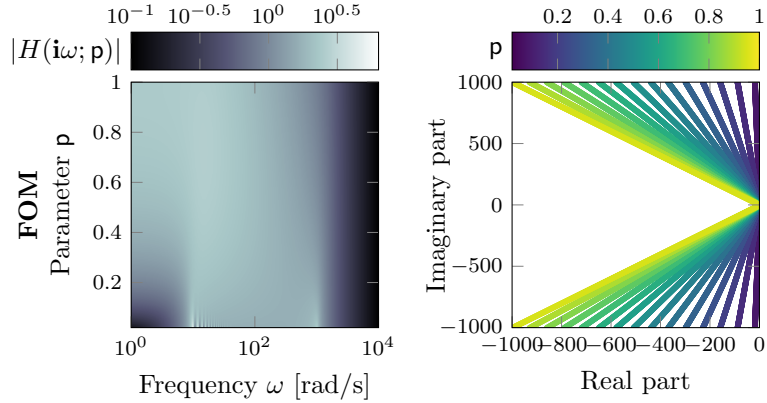

  \centering
  \surfpolefomplots{Synthetic}
  \caption{Frequency response magnitude $\abs{H(\imag \omega; \p)}$ and system
    poles of Example~1.}%
  \label{fig:surf_synthetic}
\end{figure}%
\begin{figure}[t]
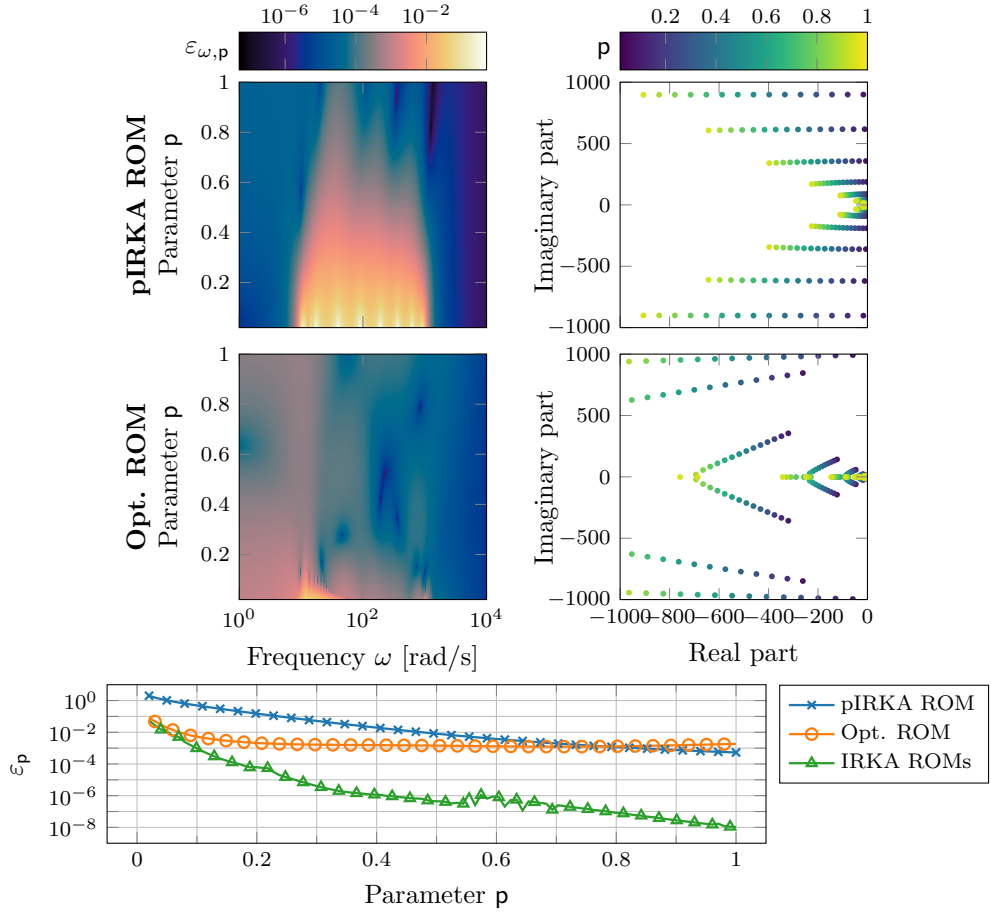

  \centering
  \surfpoleromplots{Synthetic}{4_4_16}
  \errparam{Synthetic}{4_4_16}
  \caption{Example~1 ROM performance.
  	Scaled frequency response error $\varepsilon_{\omega, \p}$ and
    poles for the pIRKA ROM (top row), the same for our optimization-derived ROM (middle row),
    and scaled $\Htwo$ error $\varepsilon_{\p}$ for the pIRKA ROM, our optimized-derived ROM,
    and various pointwise non-parametric IRKA-derived ROMs (bottom).}%
  \label{fig:surf_h2l2_synthetic}
\end{figure}%
\begin{figure}[t]
  \centering
  \surfpolefomplots{Penzl}
  \caption{Frequency response magnitude $\abs{H(\imag \omega; \p)}$ and system
    poles of Example~2.}%
  \label{fig:surf_penzl}
\end{figure}%
\begin{figure}[t]
  \centering
  \surfpoleromplots{Penzl}{3_4_12}
  \errparam{Penzl}{3_4_12}
  \caption{Example~2 ROM performance.
    See the caption of \cref{fig:surf_h2l2_synthetic} for more details.}%
  \label{fig:surf_h2l2_penzl}
\end{figure}%
\begin{figure}[t]
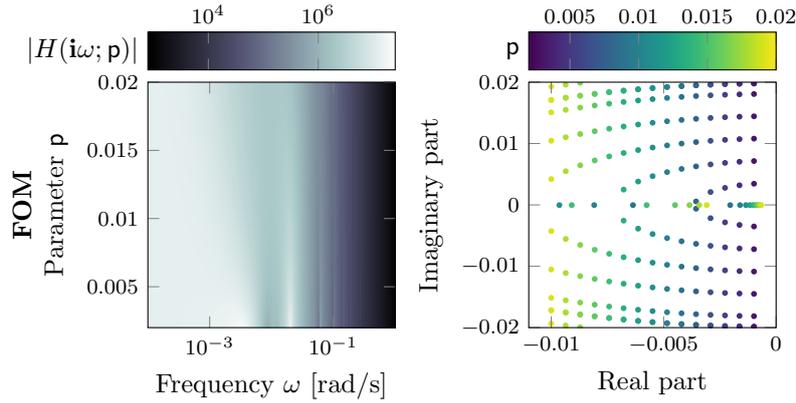

  \centering
  \surfpolefomplots{TripleChainParametricDamping}
  \caption{Frequency response magnitude $\abs{H(\imag \omega; \p)}$ and a subset of the system
    poles (ones near the origin) of Example~3.}%
  \label{fig:surf_triple_chain}
\end{figure}%
\begin{figure}[t]
  \centering
  \surfpoleromplotsextended{TripleChainParametricDamping}{TripleChainParametricDampingBCpIRKA}{TripleChainParametricDampingEABCpIRKA}{2_2_4}
  \errparamextended{TripleChainParametricDamping}{TripleChainParametricDampingBCpIRKA}{TripleChainParametricDampingEABCpIRKA}{2_2_4}
  \caption{Example~3 ROM performance.
    See the caption of \cref{fig:surf_h2l2_synthetic} for more details.}%
  \label{fig:surf_h2l2_triple_chain}
\end{figure}%
experiments were done in \matlab{} R2019b on a computer with two Intel Xeon
Silver $4110$ CPUs ($8$ cores per CPU) and $192$~GB of RAM\@.
Running times were measured using \texttt{tic} and \texttt{toc}.
We implemented our new algorithm using \granso:
GRadient-based Algorithm for Non-Smooth Optimization.\footnote{Available at
\url{http://www.timmitchell.com/software/GRANSO/}.}
We used GRANSO because
(a) it supports custom stopping criteria,
which we used to implement~\eqref{eq:our_conv_crit}, and
(b) when no constraints are given explicitly (per our implementation description
in \cref{sec:opt_based}),
GRANSO does BFGS (for constrained problems, GRANSO uses the BFGS-SQP algorithm
of~\cite{CurMO17}).
We used GRANSO v1.6.4 with its default parameters,
except that we set $\mathtt{opts.maxit} = 250$ and
$\mathtt{opts.opt\_tol} = 0$.
We set the latter parameter to zero since we only want to halt optimization of
the ROM once~\eqref{eq:our_conv_crit} is satisfied;
for this custom stopping condition,
we used $\mathtt{tol} = 10^{-5}$.
Our software,
used to compute all results reported here,
is open source.\footnote{DOI\@:
\href{https://doi.org/10.5281/zenodo.5710777}{10.5281/zenodo.5710777}}

For validating our method,
we considered one-parameter FOMs of the form
\begin{align}
  \label{eq:pLTI_A_sep}
  \begin{split}
    E \dot{x}(t; \p) & = \myparen{A_1 + \p A_2} x(t; \p) + B u(t), \\
    y(t; \p) & = C x(t; \p).
  \end{split}
\end{align}
This choice allowed us to use projection-based methods for both initialization
and as a comparison to our optimization-based approach.
We refer to this alternative as pIRKA (``piecewise IRKA''),
as it is based on the piecewise $\Htwo$-optimal interpolatory parametric
MOR method~\cite[Algorithm~5.1]{morBauBBetal11};
pIRKA consists of:
\begin{enumerate}
\item choosing $p_s$ linearly-spaced parameter values
  $\p^{(1)}, \dots, \p^{(p_s)} \in \pset$,
\item\label{item:pIRKA-IRKA} reducing local non-parametric models
  $H(\cdot; \p^{(i)})$ using IRKA to obtain local basis matrices
  $V^{(i)}, W^{(i)} \in \bbR^{n \times r_s}$ for some $r_s < n$,
\item choosing a global basis matrix $V$ as the first $r \le 2 p_s r_s$ left
  singular vectors of
  \begin{equation*}
    \begin{bmatrix}
      V^{(1)} & \cdots & V^{(p_s)} & W^{(1)} & \cdots & W^{(p_s)}
    \end{bmatrix},
  \end{equation*}
\item\label{item:pIRKA-proj} projecting full-order matrices
  \begin{align}
  \label{eq:matrices_pIRKA}
    \chE = V\tran E V, \quad
    \chA_1 = V\tran A_1 V, \quad
    \chA_2 = V\tran A_2 V, \quad
    \chB = V\tran B, \quad
    \chC = C V.
  \end{align}
\end{enumerate}
In step~\ref{item:pIRKA-IRKA},
we used \texttt{mess\_tangential\_irka} from the
Matrix Equation Sparse Solver (\mmess{}) library~\cite{SaaKB19-mmess-2.0}
with the following settings in the \texttt{opts.irka} structure:
$\mathtt{r} = r_s$,
$\mathtt{maxiter} = 100$,
$\mathtt{h2\_tol} = \mathtt{shift\_tol} = 10^{-6}$, and
$\mathtt{flipeig} = 0$.
We used a one-sided projection in step~\ref{item:pIRKA-proj}
because for our examples here, which are all uniformly strictly dissipative
(for all $\p \in \pset$, $\Ap + \Ap\tran$ and $\Ep$ are, respectively, negative
definite and positive definite),
it guarantees asymptotic stability of the ROM for all parameters.

To compare accuracies, for each ROM,
we computed the $\HoL$ error once using~\eqref{eq:HoL_norm_gramian} and
\texttt{mess\_lradi} from \mmess{}, so that the large Lyapunov
equation in~\eqref{eq:controllability_large} or~\eqref{eq:observability_large}
could be efficiently solved.
For conciseness,
we introduce:
\begin{equation*}
  \varepsilon
  = \frac{\normHoL[\big]{H - \hH}}{\normHoL*{H}}, \
  \varepsilon_{\p}
  = \frac{\normHtwo[\big]{H(\cdot; \p) - \hH(\cdot; \p)}}{\normHoL*{H}}, \
  \varepsilon_{\omega, \p}
  = \frac{\normF[\big]{H(\imag\omega; \p)
      - \hH(\imag\omega; \p)}}{\normHoL*{H}}.
\end{equation*}
In order, $\varepsilon$ is the relative $\HoL$ error,
$\varepsilon_{\p}$ is the $\Htwo$ error for a particular parameter value $\p$
relative to the $\HoL$ norm of the original system,
while finally,
$\varepsilon_{\omega, \p}$ is the transfer function value error for a particular
frequency $\omega$ and parameter value $\p$ relative to the $\HoL$ norm of the
original system.

\subsection{Example 1: A synthetic parametric model}%
\label{sec:example_synth}
Our first example is the synthetic parametric model~\cite{morwiki_synth_pmodel},
where we used $n = 1000$ for the order of the FOM and $\pset = [0.02, 1]$
as the parameter set.
From \cref{fig:surf_synthetic},
we see that the system poles move closer to the imaginary axis as $\p$
decreases.
For computing ROMs, we chose $r = 16$ (so $800$ optimization variables) and
used $p_s = 4$ and $r_s = 4$ for pIRKA\@.

The initial ROM computed by pIRKA resulted in a relative error $\varepsilon$ of
$0.3145$.
Meanwhile, optimizing the ROM using \granso{} took $250$~iterations
($17.3$~minutes), where the relative error $\varepsilon$ was
reduced to $8.395 \times 10^{-3}$.
In other words, our new approach produced a ROM whose $\HoL$ error is about $37$
times better than the one obtained by pIRKA alone.
\Cref{fig:surf_h2l2_synthetic} shows the behavior of ROMs from pIRKA and
\granso{}.
On the top left,
bright peaks can be seen for the pIRKA, i.e.,
regions where the ROM does not approximate the FOM well.
On the middle right,
the errors overall are more homogeneous and smaller.
Comparing ROM poles (top right and middle right),
we see that poles for the pIRKA ROM and FOM move horizontally in a similar
fashion with respect to $\p$,
while GRANSO found a ROM with a more complex parameter dependency of the poles.
The error improvement is also depicted in the bottom graph,
where we observe that optimization produced a ROM that is better for most
parameter values in $\pset$ and up to two orders of magnitude for some.
To get a sense for how close our ROM computed by \granso{} is to optimal
\emph{unstructured} ROMs (where $\hE, \hA, \hB, \hC$ are arbitrary matrices),
the figure also shows the results from running IRKA for individual parameter
values.
Note that this is only an indication for the lower bound,
since IRKA can not be guaranteed to find a global minimum
(the jumps in the curve confirm this).
We see that \granso{}'s ROM appears to be close to the optimal ROMs for small
parameter values, while the error flattens for larger ones.

\subsection{Example 2: A parametric version of Penzl's FOM}%
\label{sec:example_penzl}
We now consider the parameterized version of the Penzl FOM~\cite{morPen99}
used in~\cite{morIonA14}, where $n = 1006$ and $\pset = [10, 100]$.
As illustrated in \cref{fig:surf_penzl},
changing the parameter moves one of the complex conjugate pairs of poles.
Here we used $r = 12$ (so $456$ optimization variables) and
$p_s = 3$ and $r_s = 4$ for pIRKA\@.

For this FOM, pIRKA produced a ROM with a relative error $\varepsilon$ of
$2.574 \times 10^{-2}$.
In contrast,
our optimization-based algorithm reduced that error to
$6.051 \times 10^{-4}$,
i.e., about $43$ times better, which took $70$ optimization iterations
($6.8$~minutes).
From \cref{fig:surf_h2l2_penzl},
it seems that the major factor determining the quality of the ROMs was how well
the middle peak was captured.
We also see that there is more than a 1.5 orders of magnitude improvement in
\granso{}'s ROM for all parameter values.
The errors for the unstructured ROMs are better but by less than an order of magnitude.

\subsection{Example 3: A parametric damped linear vibrational mechanical system}%
\label{sec:example_triple_chain}
Finally, we consider a parametric version of the triple chain example~\cite{TruV09}
\begin{align}
  \label{eq:second_order}
  \begin{split}
    \cM \ddot{x}(t; \p) + \cDp \dot{x}(t; \p) + \cK x(t; \p)
    & = \cB u(t), \\
    y(t; \p)
    & = \cC x(t; \p),
  \end{split}
\end{align}
a second-order system with $\cM, \cDp, \cK \in \bbR^{\tilde{n} \times \tilde{n}}$,
$\cB, \cC\tran \in \bbR^{\tilde{n} \times 1}$, and
$\tilde{n} = 1501$.
The damping matrix is $\cD = \alpha \cM + \beta \cK$ with
scalars $\alpha, \beta > 0$,
so by choosing $\p = \alpha = \beta$,
we have the parameter dependency $\cDp = \p (\cM + \cK)$.
\Cref{fig:surf_triple_chain} shows the frequency response of the system and
its poles near the origin;
as expected,
the poles move closer to the imaginary axis as damping is decreased.
Using the procedure described in~\cite{PanWL12},
the second-order system~\eqref{eq:second_order} can be transformed into a
strictly dissipative first-order realization of the form~\eqref{eq:pLTI_A_sep}
with order $n = 2 \tilde{n} = 3002$ and matrices
\begin{align*}
    E
    & =
      \begin{bsmallmatrix}
        \cK & \gamma \cM \\
        \gamma \cM & \cM
      \end{bsmallmatrix},\ \
    A_1
    =
      \begin{bsmallmatrix}
        -\gamma \cK & \cK \\
        -\cK & \gamma \cM
      \end{bsmallmatrix},\ \
    A_2
    =
      \begin{bsmallmatrix}
        0 & -\gamma (\cM + \cK) \\
        0 & -(\cM + \cK)
      \end{bsmallmatrix},\ \
    B
    =
      \begin{bsmallmatrix}
        \gamma \cB \\
        \cB
      \end{bsmallmatrix},\ \
    C
    =
      \begin{bsmallmatrix}
        \cC & 0
      \end{bsmallmatrix},
\end{align*}%
where
$0
< \gamma
< \min_{\p \in \pset} \eig{\cDp, \cM + \frac{1}{4} \cDp \cK^{-1} \cDp}$.
For our experiment, we used
$\pset = [2 \times 10^{-3}, 2 \times 10^{-2}]$ and
$\gamma
= \frac{1}{2} \min_{\p \in \pset}
\eig{\cDp, \cM + \frac{1}{4} \cDp \cK^{-1} \cDp}$,
where $\gamma$ was computed using Chebfun.
Note that the transformation to a strictly dissipative form is only needed for
pIRKA to guarantee an asymptotically stable ROM\@.
The optimization procedure can work with other first-order realizations of the FOM\@.

In this third experiment, we now explore the effects of choosing ROMs
with different parameter structures, specifically
\begin{center}
  \renewcommand{\arraystretch}{1.5}
  \begin{tabular}{l|c|c|c|c}
    Variant            &  $\hEp$            &  $\hAp$             &  $\hBp$             &  $\hCp$  \\
    \hline
    ROM$_{\text{SP}}$  &  $\hE$             &  $\hA_1 + \p \hA_2$ &  $\hB$              &  $\hC$  \\
    ROM$_{\text{IO}}$  &  $\hE$             &  $\hA$              &  $\hB_1 + \p \hB_2$ &  $\hC_1 + \p \hC_2$  \\
    ROM$_{\text{All}}$ &  $\hE_1 + \p\hE_2$ &  $\hA_1 + \p \hA_2$ &  $\hB_1 + \p \hB_2$ &  $\hC_1 + \p \hC_2$
  \end{tabular}
\end{center}
where ROM$_{\text{SP}}$ denotes a structure-preserving ROM, ROM$_{\text{IO}}$
denotes that the input and output matrix-valued functions are parametric,
and ROM$_{\text{All}}$ denotes that all matrix-valued functions are parametric.
Using the matrices~\eqref{eq:matrices_pIRKA} obtained via pIRKA,
we initialized our algorithm for each variant as follows
\begin{center}
  \renewcommand{\arraystretch}{1.5}
  \begin{tabular}{l|cc|cc|cc|cc}
    Variant            &  $\hE_1$  &  $\hE_2$ &  $\hA_1$ &  $\hA_2$  &  $\hB_1$  &  $\hB_2$ &  $\hC_1$ &  $\hC_2$  \\
    \hline
    ROM$_{\text{SP}}$  &  $\chE$   & $-$ &  $\chA_1$                    & $ \chA_2$   &  $\chB$ & $-$  &  $\chC$ & $-$ \\
    ROM$_{\text{IO}}$  &  $\chE$   & $-$ &  $\chA_1 + 0.011 \cdot \chA_2$ & $-$         &  $\chB$ & $0$ &  $\chC$ & $0$ \\
    ROM$_{\text{All}}$ &  $\chE $  & $0$ &  $\chA_1$                    & $ \chA_2$   &  $\chB$ & $0$ &  $\chC$ & $0$
  \end{tabular}
\end{center}
where $-$ denotes matrices that are not applicable for
the given parametric ROM, and for ROM$_{\text{IO}}$, the value $0.011$ is chosen as
the midpoint of the parameter interval $\pset$.
We chose reduced order $r = 4$
and set $p_s = 2$ and $r_s = 2$ for pIRKA\@.

For our triple chain example, pIRKA produced a ROM, where the relative error
was $5.267 \times 10^{-2}$.
\cref{table:results} shows performance data for our three different parametric ROMs.
\begin{table}
  \renewcommand{\arraystretch}{1.2}
  \centering
  \caption{For Example~3, we show a comparison of the performance of different
    parametric ROMs (Variant) with respect to
    the number of optimization variables ($\#$ of var.),
    the number of iterations ($\#$ of iter.),
    the runtime in minutes (Time (min.)) and
    the relative error $\varepsilon$ (Rel. Err. $\varepsilon$).
  }%
  \label{table:results}
  \begin{tabular}{l|c|c|c|c}
    Variant & $\#$ of var. &$\#$ of iter. & Time (min.) & Rel.\ err.\ $\varepsilon$ \\
    \hline
    ROM$_{\text{SP}}$  & $56$ & $225$ & $57.2$  & $2.506 \times 10^{-2}$ \\
    ROM$_{\text{IO}}$  & $48$ & $135$ & $470.4$ & $3.901 \times 10^{-1}$\\
    ROM$_{\text{All}}$ & $80$ & $212$ & $730.4$ & $2.506 \times 10^{-2}$\\
  \end{tabular}
\end{table}
As can be seen, the approximation quality of ROM$_{\text{SP}}$ is $2.1$ times better
than the result obtained by pIRKA\@.
In~\cref{fig:surf_h2l2_triple_chain},
we see that the error is mostly in the lower frequencies.
Meanwhile, ROM$_{\text{IO}}$ performed worse than pIRKA, which is not surprising since this variant
cannot model changes in the poles.
Finally, for ROM$_{\text{All}}$, we see that the relative error $\varepsilon$
is the same as for ROM$_{\text{SP}}$.
This too is not surprising as ROM$_{\text{SP}}$ is very close to
the optimal unstructured models for most of the parameter values
and the additional parametrizations available in ROM$_{\text{All}}$ do not capture any structure
in the FOM that ROM$_{\text{SP}}$ does not.
Note that the runtimes for the ROM$_{\text{IO}}$ and ROM$_{\text{All}}$ variants are both higher than
the runtime for the ROM$_{\text{SP}}$ variant;
this is because the \texttt{integral} function in \matlab{} required significantly more quadrature points for these variants
than it did for ROM$_{\text{SP}}$.
\afterpage{\FloatBarrier}



\section{Concluding remarks}\label{sec:conclusion}
In our new MOR method using gradients of the $\HoL$ error,
solving the numerous sparse-dense Sylvester equations
is the vast majority of the overall cost.
As such, faster methods for parametric matrix equations could make our method even more efficient,
and so, leveraging reduced basis approaches or tensor techniques,
e.g.,~\cite{morSonS17,KrePT14}, are promising research directions.
Moreover, it could be interesting to investigate the numerical
stability issues in parametric TSIA or alternative ROM representations that
might eliminate the need for inverses.



\begin{thebibliography}{10}

\bibitem{morAnt05}
{\scshape A.~C. Antoulas}, {\em Approximation of Large-Scale Dynamical
  Systems}, vol.~6 of Adv. Des. Control, {SIAM} Publications, 2005,
  \url{https://doi.org/10.1137/1.9780898718713}.

\bibitem{BarS72}
{\scshape R.~H. Bartels and G.~W. Stewart}, {\em Solution of the matrix
  equation ${AX}+{XB}={C}$: {A}lgorithm 432}, Comm. {ACM}, 15 (1972),
  pp.~820--826.

\bibitem{morBauBBetal11}
{\scshape U.~Baur, C.~A. Beattie, P.~Benner, and S.~Gugercin}, {\em
  Interpolatory projection methods for parameterized model reduction}, {SIAM}
  J. Sci. Comput., 33 (2011), pp.~2489--2518,
  \url{https://doi.org/10.1137/090776925}.

\bibitem{morBauBHetal17}
{\scshape U.~Baur, P.~Benner, B.~Haasdonk, C.~Himpe, I.~Martini, and
  M.~Ohlberger}, {\em Comparison of methods for parametric model order
  reduction of time-dependent problems}, in Benner et~al.
  \cite{morBenOCetal17}, pp.~377--407,
  \url{https://doi.org/10.1137/1.9781611974829.ch9}.

\bibitem{morBenGW15}
{\scshape P.~Benner, S.~Gugercin, and K.~Willcox}, {\em A survey of projection-based model
  reduction methods for parametric systems}, SIAM Review, 57 (2015),
  pp.~483--531, \url{https://doi.org/10.1137/130932715}.

\bibitem{morBenKS11}
{\scshape P.~Benner, M.~K{\"o}hler, and J.~Saak}, {\em Sparse-dense {S}ylvester
  equations in ${H}_2$-model order reduction}, Preprint MPIMD/11-11, Max Planck
  Institute Magdeburg, Dec. 2011.

\bibitem{morBenOCetal17}
{\scshape P.~Benner, M.~Ohlberger, A.~Cohen, and K.~Willcox}, eds., {\em Model
  Reduction and Approximation: Theory and Algorithms}, SIAM, Philadelphia, PA,
  2017, \url{https://doi.org/10.1137/1.9781611974829}.

\bibitem{Chu87}
{\scshape E.~K.-w. Chu}, {\em The solution of the matrix equations
  {$AXB-CXD=E$} and {$(YA-DZ,YC-BZ)=(E,F)$}}, Linear Algebra Appl., 93 (1987),
  pp.~93--105, \url{https://doi.org/10.1016/S0024-3795(87)90314-4}.

\bibitem{Col12}
{\scshape R.~Coleman}, {\em Calculus on Normed Vector Spaces}, vol.~1 of
  Universitext, Springer-Verlag, New York, 2012,
  \url{https://doi.org/10.1007/978-1-4614-3894-6}.

\bibitem{CurMO17}
{\scshape F.~E. Curtis, T.~Mitchell, and M.~L. Overton}, {\em A {BFGS-SQP}
  method for nonsmooth, nonconvex, constrained optimization and its evaluation
  using relative minimization profiles}, Optim. Methods Softw., 32 (2017),
  pp.~148--181, \url{https://doi.org/10.1080/10556788.2016.1208749}.

\bibitem{DriHT14}
{\scshape T.~A. Driscoll, N.~Hale, and L.~N. Trefethen}, {\em Chebfun Guide},
  Pafnuty Publications, 2014.
\newblock \url{http://www.chebfun.org/docs/guide/}.

\bibitem{Fol17}
{\scshape G.~B. Folland}, {\em Real Analysis: Modern Techniques and their
  Applications}, John Wiley \& Sons, Inc., 1999.

\bibitem{morGosGU21}
{\scshape I.~V. Gosea, S.~Gugercin, and B.~Unger}, {\em Parametric model
  reduction via rational interpolation along parameters}, e-print 2104.01016,
  arXiv, 2021, \url{https://arxiv.org/abs/2104.01016}.
\newblock math.NA.

\bibitem{morGri18}
{\scshape A.~R. Grimm}, {\em Parametric Dynamical Systems: Transient Analysis
  and Data Driven Modeling}, PhD thesis, Virginia Polytechnic Institute and
  State University, 2018, \url{http://hdl.handle.net/10919/83840}.

\bibitem{morGugAB08}
{\scshape S.~Gugercin, A.~C. Antoulas, and C.~Beattie}, {\em {$\mathcal{H}_2$}
  model reduction for large-scale linear dynamical systems}, {SIAM} J. Matrix
  Anal. Appl., 30 (2008), pp.~609--638,
  \url{https://doi.org/10.1137/060666123}.

\bibitem{morHesRS16}
{\scshape J.~S. Hesthaven, G.~Rozza, and B.~Stamm}, {\em Certified {R}educed
  {B}asis {M}ethods for {P}arametrized {P}artial {D}ifferential {E}quations},
  SpringerBriefs in Mathematics, Springer, Cham, 2016,
  \url{https://doi.org/10.1007/978-3-319-22470-1}.

\bibitem{morHunMS18}
{\scshape M.~Hund, P.~Mlinari\'{c}, and J.~Saak}, {\em An \(\mathcal{H}_2
  \otimes \mathcal{L}_2\)-optimal model order reduction approach for parametric
  linear time-invariant systems}, Proc. Appl. Math. Mech., 18 (2018),
  p.~e201800084, \url{https://doi.org/10.1002/pamm.201800084}.

\bibitem{morIonA14}
{\scshape A.~C. Ionita and A.~C. Antoulas}, {\em Data-driven parametrized model
  reduction in the {L}oewner framework}, {SIAM} J. Sci. Comput., 36 (2014),
  pp.~A984--A1007, \url{https://doi.org/10.1137/130914619}.

\bibitem{KrePT14}
{\scshape D.~Kressner, M.~Ple{\v s}inger, and C.~Tobler}, {\em A preconditioned
  low-rank {CG} method for parameter-dependent {L}yapunov matrix equations},
  Numer. Lin. Alg. Appl., 21 (2014), pp.~666--684,
  \url{https://doi.org/10.1002/nla.1919}.

\bibitem{morMeiL67}
{\scshape L.~Meier and D.~G. Luenberger}, {\em Approximation of linear constant
  systems}, {IEEE} Trans. Autom. Control, 12 (1967), pp.~585--588,
  \url{https://doi.org/10.1109/TAC.1967.1098680}.

\bibitem{morMli20}
{\scshape P.~Mlinari{\'c}}, {\em Structure-preserving model order reduction for
  network systems}, {Dissertation}, Department of Mathematics, Otto von
  Guericke University, Magdeburg, Germany, 2020,
  \url{https://doi.org/10.25673/33570}.

\bibitem{PanWL12}
{\scshape H.~Panzer, T.~Wolf, and B.~Lohamnn}, {\em A strictly dissipative
  state space representation of second order systems},
  at-Auto\-mati\-sie\-rungs\-tech\-nik, 60 (2012), pp.~392--397,
  \url{https://doi.org/10.1524/auto.2012.1015}.

\bibitem{morPen99}
{\scshape T.~Penzl}, {\em Algorithms for model reduction of large dynamical
  systems}, Technical report SFB393/99-40, SFB 393 {\itshape Numerische
  Simulation auf massiv parallelen Rechnern}, TU Chem\-nitz, 1999.
\newblock Available from \url{http://www.tu-chemnitz.de/sfb393/sfb99pr.html}.

\bibitem{morPet13}
{\scshape D.~Petersson}, {\em A Nonlinear Optimization Approach to
  {$\mathcal{H}_{2}$}-Optimal Modeling and Control}, dissertation, Link\"oping
  University, 2013,
  \url{http://liu.diva-portal.org/smash/get/diva2:647068/FULLTEXT01.pdf}.

\bibitem{morQuaMN16}
{\scshape A.~Quarteroni, A.~Manzoni, and F.~Negri}, {\em Reduced Basis Methods
  for Partial Differential Equations}, vol.~92 of La Matematica per il 3+2,
  Springer International Publishing, 2016.
\newblock ISBN: 978-3-319-15430-5.

\bibitem{SaaKB19-mmess-2.0}
{\scshape J.~Saak, M.~K\"{o}hler, and P.~Benner}, {\em {M-M.E.S.S.-2.0} -- the
  matrix equations sparse solvers library}, Aug. 2019,
  \url{https://doi.org/10.5281/zenodo.3368844}.

\bibitem{morSonS17}
{\scshape N.~T. Son and T.~Stykel}, {\em Solving parameter-dependent {L}yapunov
  equations using the reduced basis method with application to parametric model
  order reduction}, {SIAM} J. Matrix Anal. Appl., 38 (2017), pp.~478--504,
  \url{https://doi.org/10.1137/15M1027097}.

\bibitem{morwiki_synth_pmodel}
{\scshape {The MORwiki Community}}, {\em Synthetic parametric model}.
\newblock Hosted at {MORwiki} -- Model Order Reduction Wiki, 2005,
  \url{http://modelreduction.org/index.php/Synthetic_parametric_model}.

\bibitem{TruV09}
{\scshape N.~Truhar and K.~Veseli{\'c}}, {\em An efficient method for
  estimating the optimal dampers' viscosity for linear vibrating systems using
  {L}yapunov equation}, {SIAM} J. Matrix Anal. Appl., 31 (2009), pp.~18--39,
  \url{https://doi.org/10.1137/070683052}.

\bibitem{morVanGA08}
{\scshape P.~Van~Dooren, K.~Gallivan, and P.-A. Absil}, {\em
  {$\mathcal{H}_2$}-optimal model reduction of {MIMO} systems}, Appl. Math.
  Lett., 21 (2008), pp.~1267--1273,
  \url{https://doi.org/10.1016/j.aml.2007.09.015}.

\bibitem{morWil70}
{\scshape D.~A. Wilson}, {\em Optimum solution of model-reduction problem},
  Proceedings of the Institution of Electrical Engineers, 117 (1970),
  pp.~1161--1165, \url{https://doi.org/10.1049/piee.1970.0227}.

\bibitem{morWitTKetal16}
{\scshape P.~Wittmuess, C.~Tarin, A.~Keck, E.~Arnold, and O.~Sawodny}, {\em
  Parametric model order reduction via balanced truncation with {T}aylor series
  representation}, {IEEE} Trans. Autom. Control, 61 (2016), pp.~3438--3451,
  \url{https://doi.org/10.1109/TAC.2016.2521361}.

\bibitem{morXuZ11}
{\scshape Y.~Xu and T.~Zeng}, {\em Optimal {$\mathcal{H}_2$} model reduction
  for large scale {MIMO} systems via tangential interpolation}, Int. J. Numer.
  Anal. Model., 8 (2011), pp.~174--188.

\bibitem{Zei86}
{\scshape E.~Zeidler}, {\em Nonlinear Functional Analysis and Its Applications
  I: Fixed-Point Theorems}, Springer-Verlag, New York, 1986,
  \url{https://doi.org/10.1007/BF00047050}.

\end{thebibliography}

\end{document}